\documentclass[twoside]{article}
\usepackage[utf8]{inputenc}
\usepackage[french,english]{babel}
\usepackage[margin=1in]{geometry}
\usepackage[titletoc,title]{appendix}
\usepackage{amsmath,amsfonts,amssymb,mathtools,amsthm}
\usepackage{graphicx,float}
\usepackage[cyr]{aeguill}

\usepackage{hyperref}
\usepackage{cite}
\usepackage{cleveref}
\usepackage{epsfig}
\usepackage{pinlabel}
\usepackage{subcaption}

\usepackage{tikz}
\usepackage{tikz-cd}
\usetikzlibrary{math}
\usepackage{pgf}
\usetikzlibrary{arrows}
\usetikzlibrary{shapes.geometric}
\usetikzlibrary{shapes.misc}

\usepackage{array}
\usepackage{setspace, caption, subcaption}

\usepackage{amsfonts,amssymb,amsmath,amsthm}
\usepackage{enumitem}
\usepackage[all]{xy}

\newcommand{\p}{\ensuremath{\mathbb{P}}}
\newcommand{\Z}{\ensuremath{\mathbb{Z}}}

\DeclareMathOperator{\GL}{GL}
\DeclareMathOperator{\PGL}{PGL}
\DeclareMathOperator{\Aut}{Aut}
\DeclareMathOperator{\Bir}{Bir}
\newcommand{\Pic}{{\mathrm{Pic}}}

\DeclareMathOperator{\Gal}{Gal}
\DeclareMathOperator{\Sym}{Sym}
\DeclareMathOperator{\id}{id}
\DeclareMathOperator{\D}{D}

\DeclareMathOperator{\Spec}{Spec}

\renewcommand{\k}{\mathbf{k}}

\def\dashmapsto{\mapstochar\dashrightarrow}

\newtheorem{definition}{Definition}[section]
\newtheorem{lemma}[definition]{Lemma}
\newtheorem{proposition}[definition]{Proposition}
\newtheorem{theorem}[definition]{Theorem}
\newtheorem{remark}[definition]{Remark}

\newcommand{\myinline}[1]{\raisebox{-.5\height}{{#1}}}

\author{Aurore Boitrel}
\newcommand{\Addresses}{{
  \bigskip
  \footnotesize
  \textsc{Aurore Boitrel, Aix Marseille Univ, CNRS, I2M, SMF, CIRM, Marseille, France}\par\nopagebreak
  \textit{E-mail address:} \texttt{aurore.boitrel@universite-paris-saclay.fr}
}}
\title{Automorphism groups of real rational quartic del Pezzo surfaces}

\date{}

\usepackage{tocloft}

\usepackage{fancyhdr}
\renewcommand{\thepage}{\arabic{page}}
\begin{document}

\maketitle

\begin{abstract}
In this paper we give a complete description of all possible automorphism groups of real $\mathbb{R}$-rational del Pezzo surfaces $X$ of degree $4$, using the description of $X$ as the blow-up of some smooth real quadric surface $Q$ in $\p^{3}_{\mathbb{R}}$. We examine all possible ways to blow up $4$ geometric points on $Q$, illustrate in each case the $\Gal(\mathbb{C}/\mathbb{R})$-action on the conic bundle structures on $X_{\mathbb{C}}$, and use it to give a geometric description of the real automorphism group $\Aut_{\mathbb{R}}(X)$ by generators in terms of automorphisms and birational automorphisms of $Q$. As a consequence, we get which finite subgroups of $\Bir_{\mathbb{C}}(\p^{2})$ can act faithfully by automorphisms on real $\mathbb{R}$-rational del Pezzo surfaces of degree $4$.
\end{abstract}

\tableofcontents

\definecolor{forestgreen}{rgb}{0,0.65,0}
\definecolor{cyan}{rgb}{0.871,0.494,0}

\section{Introduction}
\label{Section:Section0}

Del Pezzo surfaces, i.e. smooth projective surfaces with ample anticanonical divisor class, play an important role in the birational classification of algebraic projective surfaces, as minimal del Pezzo surfaces arise as extremal contractions in the Minimal Model Program (MMP) starting for instance from rationally connected surfaces. Del Pezzo surfaces also play a key role in the classification of finite subgroups of the plane Cremona group $\Bir_{\k}(\p^{2})$, which is the group of birational automorphisms of the projective plane over a perfect field $\k$. The general approach to this classification is the following (for more details, see \cite{di09a,bla09}): if we let $G$ be a finite subgroup of $\Bir_{\k}(\p^{2})$, we can first regularize the birational action of $G$ on a $\k$-rational smooth projective surface $X$ (see, for instance, \cite[Lemma 6]{di09b}), over which $G$ now acts biregularly. Moreover, by applying a $G$-equivariant version of the MMP (see for instance \cite{pro21}), we may further assume that $X$ admits a structure of $G$-Mori fibre space. This means in this case that either $X$ admits a structure of $G$-conic bundle over $\p^{1}_{\k}$ with $G$-invariant Picard rank $2$ or $X$ is a $G$-del Pezzo surface with $G$-invariant Picard rank $1$, and over which $G$ still acts by automorphisms.

The degree $d$ of a del Pezzo surface $X$ is the self-intersection number $K_{X}^{2}$ of its canonical class, and it is an integer between $1$ and $9$. Recall that over an algebraically closed field, a del Pezzo surface is either a quadric surface or a blow-up of $\p^{2}$ at $9-d$ points in general position (\hspace{1sp}\cite{man86}).\\
Over the complex numbers, the complete description of automorphism groups of del Pezzo surfaces of any degree can be found in \cite{di09a}, where the classification up to conjugacy of finite subgroups of $\Bir_{\mathbb{C}}(\p^{2})$ is obtained. See also \cite{bla09} for the exposition of these results. Over an algebraically closed field of positive characteristic, the classification of automorphisms of del Pezzo surfaces of small degree $d \leq 4$ has been recently completed due to \cite{dd19,dm23a,dm23b}. Much less is known when the ground field is not algebraically closed. Nevertheless, some recent (partial) results exist. More precisely, automorphism groups of real rational del Pezzo surfaces have been studied in \cite{rob16,rz18,yas22}. And when the base field is an arbitrary perfect field, rational del Pezzo surfaces of degrees $8$ and $6$ and their automorphism groups have been classified in \cite{sz21}. More recently, automorphism groups of del Pezzo surfaces of degree $5$ over an arbitrary perfect field have been independently described in \cite{zai23} and \cite{boi25}. More precisely, \cite{boi25} presents the classification of automorphism groups of all possible forms of del Pezzo surfaces $X$ of degree $5$ defined over any perfect field $\mathbf{k}$ in explicit geometric terms. In this article, we describe automorphism groups of real rational del Pezzo surfaces of degree $4$. Moreover, maximal automorphism groups of quartic del Pezzo surfaces over a field of characteristic zero have been described in \cite{smi23}; and an extension of the results of the latter article to smooth cubic surfaces can be found in \cite{smi24}. Furthermore, the largest automorphism group of a smooth cubic surface over an arbitrary field has been described in \cite{vik25}.

As a completion of the results of \cite{rob16,yas22}, we describe in this article all possible automorphism groups of real $\mathbb{R}$-rational del Pezzo surfaces of degree $4$. Precisely, we realize each rational real form $X$ of del Pezzo surfaces of degree $4$ as the blow-up of a smooth quadric surface $Q$ in $\p^{3}_{\mathbb{R}}$ and describe the surfaces in families by parametrizing the (four) geometric blown-up points. In each case, we illustrate the $\Gal(\mathbb{C}/\mathbb{R})$-action on the conic bundle structures of $X_{\mathbb{C}}$ and use it to describe geometrically the group $\Aut_{\mathbb{R}}(X)$ by generators in terms of automorphisms and birational automorphisms of $Q$, starting from the description of the group structure of $\Aut(X_{\mathbb{C}})$ given in \cite[Proposition 8.1.3]{bla06}. The action of $\Aut(X_{\mathbb{C}})$ on the five pairs of conic bundle structures on $X_{\mathbb{C}}$ yields a group homomorphism $\rho_{1} : \Aut(X_{\mathbb{C}}) \rightarrow \Sym_{5}$. We denote by $\rho$ the restriction of $\rho_{1}$ on the real automorphism group $\Aut_{\mathbb{R}}(X)$, and we set $A_{0}:=\text{Ker}(\rho)$ and $A':=\text{Im}(\rho) \subset \Sym_{5}$. This gives the following associated short exact sequence \begin{equation} 1 \rightarrow A_{0} \rightarrow \Aut_{\mathbb{R}}(X) \overset{\rho}{\rightarrow} A' \rightarrow 1. \end{equation} Note that we can define a homomorphism of groups $ A_{0} \overset{\gamma}{\hookrightarrow} \lbrace (a_{1},\dots,a_{5}) \in (\Z/2\Z)^{5} \, \vert \, \sum_{i=1}^{5} a_{i} = 0 \rbrace \cong (\Z/2\Z)^{4}$, where $\gamma(g)=(a_{1},\dots,a_{5})$ is such that $a_{i}=1$ if $g$ permutes the two conic bundles in the corresponding $i$-th pair, and $a_{i}=0$ if $g$ fixes both of them. In what follows we give a description of the elements of the groups $A_{0}$ and $A'$ for each rational real form $X$ of del Pezzo surfaces of degree $4$. We denote by $Q_{2,2}$ and $Q_{3,1}$ the smooth quadric surfaces given by $ \lbrace x^2 + y^2 - z^2 - w^2 = 0 \rbrace $  and $ \lbrace x^2 + y^2 + z^2 - w^2 = 0 \rbrace $ in $ \mathbb{P}^{3}_{\mathbb{R}} $, respectively. These are exactly the two non-isomorphic real rational models of $\p^{1}_{\mathbb{C}} \times \p^{1}_{\mathbb{C}}$. For a real del Pezzo surface $S$, we denote by $S(a,b)$ the blow-up of $S$ at $a$ real points and $b$ pairs of complex conjugate points. Note that the case of $X\cong Q_{3,1}(0,2)$ has been entirely treated in \cite[Lemma 4.8 and Propositions 4.9,4.10]{rob16}.
Since the multiplicity of the behaviors of the different classes appearing in this work makes it hard to encapsulate the results in a short way, let us summarizing our main result with the help of Table \ref{Table:Table1_summarising_main_result} below. We will give the corresponding precise statements together with their proofs in Section \ref{sec:Section_2_Chap_4}.

\begin{theorem} Let $X$ be a real rational del Pezzo surface of degree $4$. Then $X$ is one of the five surfaces occuring in Table \ref{Table:Table1_summarising_main_result}, where the groups $A_{0}$ and $A'$ are described in each case.
\begin{table}[!h]
\centering
\normalsize
\renewcommand*{\arraystretch}{2}
\begin{tabular}{|>{\centering\arraybackslash}m{1.5cm}|>{\centering\arraybackslash}m{2cm}|>{\centering\arraybackslash}m{2.7cm}|>{\centering\arraybackslash}m{1.5cm}|>{\centering\arraybackslash}m{5.8cm}|}
  \hline
  $X$ &  blown-up points & parameter(s) & $A_{0}$ & $A'$ \\
  \hline
  $ Q_{3,1}(4,0)$ & $p, q, r, s$ & $\lambda \in \mathbb{C}\setminus\mathbb{R}$ & $(\Z/2\Z)^{3}$ & $\langle (123),(12)(45) \rangle$, $\langle (12)(45) \rangle $ or $\lbrace \id \rbrace$ \\ 
  \hline
  $ Q_{3,1}(2,1)$ & $p, q, (r, \overline{r})$ & $\mu \in \mathbb{C} \setminus \lbrace 0,\pm1 \rbrace$ & $(\Z/2\Z)^{2}$ & $\langle (23)(45) \rangle$ or $\lbrace \text{id} \rbrace$ \\ 
  \hline
  $Q_{3,1}(0,2)$ & $(p, \overline{p}), (q, \overline{q})$ & $\mu \in \mathbb{C}\setminus\lbrace 0,\pm1 \rbrace$ & $(\Z/2\Z)^{3}$ & $\langle (23)(45) \rangle$ or $\lbrace \text{id} \rbrace$ \\ 
  \hline
  $ Q_{2,2}(4,0)$ & $p, q, r, s$ & $\mu_{1}, \mu_{2} \, {\in \mathbb{R} \setminus \lbrace 0,1} \rbrace$ & $(\Z/2\Z)^{4}$ & $\langle (13245) , (12)(45) \rangle$, $\langle (12)(45) \rangle$ or $\lbrace \id \rbrace$ \\
\hline
\end{tabular}
\caption{Summary table of the description of the groups $A_{0}$ and $A'$ for each rational real form of del Pezzo surfaces $X$ of degree $4$.}
\label{Table:Table1_summarising_main_result}
\end{table}
\label{Theorem:Theorem_main_result} 
\end{theorem} 

This paper is organized as follows. In Section \ref{sec:Section_1_Chap_4} we recall some basic facts about real quadric surfaces, as well as the description of the real automorphism groups of some of them, that we will need and use in Section \ref{sec:Section_2_Chap_4}. Section \ref{sec:Section_2_Chap_4} is devoted to the case-by-case proof of Theorem \ref{Theorem:Theorem_main_result}, with which we extend the analysis of \cite{rob16,yas22}.$\,$\\

\noindent \textbf{Acknowledgements}\\
I am very grateful to my PhD advisors Andrea Fanelli and Susanna Zimmermann for their guidance and constant support throughout this work. I would like to thank Frédéric Mangolte and Ivan Cheltsov for interesting discussions which gave me more insight on this work, as well as the Centre International de Rencontres Mathématiques (CIRM) and the Institut de Mathématiques de Marseille (I2M) for providing a good working environment during postdoc. This work was supported by the French National Centre for Scientific Research (CNRS), the University of Angers, the University Paris-Saclay and the ERC Saphidir.

\section{Preliminaries: A quick look at real quadric surfaces}
\label{sec:Section_1_Chap_4}

We start by defining the notion of conic bundle structure, which will appear several times throughout this article.

\begin{definition}[Conic bundle]
A real rational smooth surface $S$ admits a conic bundle structure if there exists a surjective morphism $ \pi : S \rightarrow \p^{1}_{\mathbb{R}} $ with connected fibres such that $-K_{S}$ is $\pi$-ample; the general geometric fibre of $\pi$ is isomorphic to $\p^{1}_{\mathbb{C}}$ and a geometric singular fibre of $\pi$ is the union of two secant $(-1)$-curves over $\mathbb{C}$.
\end{definition}

We denote by $Q_{r,s}$ the smooth quadric hypersurface $\lbrace [x_{1}:\dots:x_{r+s}] : x_{1}^{2}+\dots+x_{r}^{2}-x_{r+1}^{2}-\dots-x_{r+s}^{2}=0 \rbrace \subset \p^{r+s-1}_{\mathbb{R}}$, and for a real del Pezzo surface $S$ recall that we denote by $S(a,b)$ the blow-up of $S$ at $a$ real points and $b$ pairs of complex conjugate points. We shall mostly use $\p^{2}_{\mathbb{R}}$, $Q_{3,1}$ or $Q_{2,2}$ as $S$.\\
These last two, namely the quadric surfaces given by $ \lbrace x^2 + y^2 - z^2 - w^2 = 0 \rbrace $  and $ \lbrace x^2 + y^2 + z^2 - w^2 = 0 \rbrace $ in $ \mathbb{P}^{3}_{\mathbb{R}} $ are exactly the two non-isomorphic real rational models of $\p^{1}_{\mathbb{C}} \times \p^{1}_{\mathbb{C}}$. The first is $\mathbb{R}$-isomorphic to $ \mathbb{P}_{\mathbb{R}}^{1} \times \mathbb{P}_{\mathbb{R}}^{1} $ and the second is the $\mathbb{R}$-form given by $ \sigma : ([u_{0}:u_{1}],[v_{0}:v_{1}]) \mapsto ([\overline{v_{0}}:\overline{v_{1}}],[\overline{u_{0}}:\overline{u_{1}}]) $, where $ \langle z \mapsto \overline{z} \rangle = \text{Gal}(\mathbb{C} / \mathbb{R}) \cong \mathbb{Z}/2 \mathbb{Z}$. Indeed, forgetting the real structure given by $\sigma$ in the latter case, the surface $(Q_{3,1})_{\mathbb{C}}$ is isomorphic to $\p^{1}_{\mathbb{C}} \times \p^{1}_{\mathbb{C}}$ via the isomorphism $$\begin{array}{lrcl}
\varphi : & (Q_{3,1})_{\mathbb{C}} & \longrightarrow & \p^{1}_{\mathbb{C}} \times \p^{1}_{\mathbb{C}} \\
    & [x:y:w:z] & \longmapsto & ([x+iy:w-z],[x+iy:w+z]) \end{array},$$ whose inverse is given by $$\begin{array}{lrcl}
\varphi^{-1} : & \p^{1}_{\mathbb{C}} \times \p^{1}_{\mathbb{C}} & \longrightarrow & (Q_{3,1})_{\mathbb{C}} \\
    & ([u_{0}:u_{1}],[v_{0}:v_{1}]) & \longmapsto & [u_{0}v_{0}+u_{1}v_{1}:(-i)(u_{1}v_{1}-u_{0}v_{0}):u_{1}v_{0}+u_{0}v_{1}:u_{1}v_{0}-u_{0}v_{1}] \end{array}.$$

Let us recall the classical description of the automorphism group structure of $\p^{1}_{\mathbb{R}} \times \p^{1}_{\mathbb{R}}$, that is given by $$\Aut_{\mathbb{R}}(Q_{2,2}) \cong \Aut_{\mathbb{R}}(\p^{1}_{\mathbb{R}} \times \p^{1}_{\mathbb{R}}) \cong (\PGL_{2}(\mathbb{R}) \times \PGL_{2}(\mathbb{R})) \rtimes \langle \tau : (x,y) \mapsto (y,x) \rangle$$ (see for instance \cite[Theorem 2.2]{BEN16} for a nice proof of this result over $\mathbb{C}$), as well as the following description of the real automorphism group of $Q_{3,1}$, that we will use several times throughout this paper.

\begin{proposition}[{\cite[Proposition 4.1]{rob16}}]
The group $\Aut_{\mathbb{R}}(Q_{3,1})$ corresponds, via the isomorphism $\varphi$, to the subgroup of the group of complex automorphisms $\Aut(\p^{1}_{\mathbb{C}} \times \p^{1}_{\mathbb{C}})$ generated by $\tau : (x,y) \mapsto (y,x)$ and by $ \mathcal{F} = \lbrace (A,\overline{A}) \, \vert \, A \in \PGL_{2}(\mathbb{C}) \rbrace$. Moreover, $\Aut_{\mathbb{R}}(Q_{3,1}) \cong \mathcal{F} \rtimes \langle \tau \rangle$.
\label{prop:Proposition_description_Aut(Q31)} 
\end{proposition}

\section{Automorphisms of real rational quartic del Pezzo surfaces}
\label{sec:Section_2_Chap_4}

In this section, we will use extensively the structure of the Picard group to describe in a fairly natural manner the group of real automorphisms of each rational real form
of quartic del Pezzo surfaces, using the action of the Galois group on the conic bundle structures of the surface.

\subsection{Classical statement over the complex numbers}
\label{Sec:subsection_classical_statement_over_the_complex}

Let $X$ be a del Pezzo surface of degree $4$ defined over $\mathbb{R}$. Then $X_{\mathbb{C}}:=X \times_{\Spec(\mathbb{R})} \Spec(\mathbb{C})$ is isomorphic to the blow-up of $\p^{2}_{\mathbb{C}}$ in $5$ points in general position. Denote by $ \pi : X \rightarrow \mathbb{P}^{2}_{\mathbb{C}}$ a blow-down morphism. Up to change of coordinates in $\p^{2}_{\mathbb{C}}$, we can assume that the points are of the form: $p_{1} = [1:0:0], p_{2} = [0:1:0], p_{3} = [0:0:1], p_{4} = [1:1:1]$ and $p_{5} = [a:b:c]$, where $p_{5}$ is a point not aligned with any two of the other four points. Notice that the projective equivalence class of the set $\lbrace p_{1}, \dots, p_{5} \rbrace \subset \mathbb{P}^{2}_{\mathbb{C}}$ determine the isomorphism class of $X_{\mathbb{C}}$.

There are exactly $16$ exceptional curves on $X_{\mathbb{C}}$, given by: $E_{1} = \pi^{-1}(p_{1}), \dots, E_{5} = \pi^{-1}(p_{5})$, the five exceptional divisors above the blown-up points $p_{1}, \dots, p_{5}$; the ten strict transforms $D_{ij}$ of the lines $L_{ij}$ passing through $p_{i}, p_{j}$ for $1 \leq i,j \leq 5$, $i \neq j$; and the strict transform $C$ of the conic $\mathcal{C}$ passing through $p_{1}, \dots, p_{5}$ on $\mathbb{P}^{2}_{\mathbb{C}}$. We note that each of these $(-1)$-curves intersects five others. Each $E_{i}$ intersects $C$ and the $4$ strict transforms of lines through $p_{i}$; the divisor $C$ intersects $E_{1}, E_{2},\dots,E_{5}$; and each divisor $D_{ij}$ intersects $E_{i}$ and $E_{j}$, and the transforms of the three lines through two of the other three points.
 
\begin{definition}[{\hspace{1sp}\cite[Definition 8.1.1]{bla06}}]
A pair $\lbrace A,B \rbrace$ of divisors of $X$ is an \textit{exceptional pair} if $A$ and $B$ are divisors, $A^{2}=B^{2}=0$ and $(A+B)=-K_{X}$. This implies that $A \cdot B=2$ and $A \cdot K_{X} = B \cdot K_{X} = -2$.
\end{definition}

Denote by $L$ the pullback by $\pi$ of a general line on $\p^{2}_{\mathbb{C}}$. The Picard group of $X_{\mathbb{C}}$ is generated by the classes $E_{1}, E_{2}, E_{3}, E_{4}, E_{5}$ and $L$. There are five exceptional pairs of divisors on $X_{\mathbb{C}}$, namely $\lbrace L-E_{i} , 2L-\sum_{j \neq i}E_{j} \rbrace = \lbrace L-E_{i} , -K_{X_{\mathbb{C}}}-L+E_{i} \rbrace$, for $i = 1, \dots, 5$. Each corresponds geometrically to the two conic bundle structures given by the linear system of lines passing through one of the five blown up points, and by the linear system of conics passing through the other four points (see \cite[Lemma 8.1.2]{bla06}).\\ We can use the notion of exceptional pairs of divisors to describe the structure of the group $\Aut(X_{\mathbb{C}})$ in a very natural manner. The action of $\Aut(X_{\mathbb{C}})$ on the pairs of conic bundles gives rise to a homomorphism of groups $\rho_{1} : \Aut(X_{\mathbb{C}}) \longrightarrow \Sym_{5}$. Denoting its image and kernel respectively by $H_{X_{\mathbb{C}}}$ and $G_{X_{\mathbb{C}}}$, we have the exact sequence \begin{equation}
1 \rightarrow G_{X_{\mathbb{C}}} \rightarrow \Aut(X_{\mathbb{C}}) \overset{\rho_{1}}{\rightarrow} H_{X_{\mathbb{C}}} \rightarrow 1,
\label{eq:suite_exacte_Aut(X_C)}
\end{equation}
and the following result.

\begin{proposition}[{Structure of $\Aut(X_{\mathbb{C}})$, \cite[Proposition 8.1.3]{bla06}}]

\begin{enumerate}
\item The exact sequence \eqref{eq:suite_exacte_Aut(X_C)} splits.
\label{it:item_(1)_Proposition_structure_Aut(X_C)}

\item We can define an isomorphism of groups $ G_{X_{\mathbb{C}}} \overset{\gamma}{\longrightarrow} \lbrace (a_{1},a_{2},a_{3},a_{4},a_{5}) \in (\Z/2\Z)^{5} \, \vert \, \sum_{i=1}^{5} a_{i} = 0 \rbrace \cong (\Z/2\Z)^{4}$, where $\gamma(g)=(a_{1},\dots,a_{5})$ is such that $a_{i}=1$ if $g$ permutes $L-E_{i}$ with $-K_{X_{\mathbb{C}}}-L+E_{i}$, and $a_{i}=0$ if $g$ fixes the two divisors.
\label{it:item_(2)_Proposition_structure_Aut(X_C)}

\item We have $ H_{X_{\mathbb{C}}} \cong \lbrace h \in \PGL(3,\mathbb{C}) \, \vert \, h(\lbrace p_{1},p_{2},p_{3},p_{4},p_{5} \rbrace) = \lbrace p_{1},p_{2},p_{3},p_{4},p_{5} \rbrace \rbrace \subset \Sym_{5}$.
\label{it:item_(3)_Proposition_structure_Aut(X_C)}

\item $\Aut(X_{\mathbb{C}}) \cong G_{X_{\mathbb{C}}} \rtimes H_{X_{\mathbb{C}}}$, where the group $H_{X_{\mathbb{C}}} \subset \Sym_{5}$ acts on $G_{X_{\mathbb{C}}}$ by permuting the five coordinates of its elements.
\label{it:item_(4)_Proposition_structure_Aut(X_C)}
\end{enumerate}
\label{prop:Structure_Aut(X_C)}
\end{proposition}

\begin{remark}
The linear system $\vert -K_{X_{\mathbb{C}}} \vert$ embeds $X_{\mathbb{C}}$ into $\p^{4}_{\mathbb{C}}$ as a complete intersection of two quadrics, and by a linear change of coordinates, the two quadrics can be chosen to be $ \sum_{i=0}^{4} x_{i}^{2} = 0$ and $ \sum_{i=0}^{4} \lambda_{i} x_{i}^{2} = 0$ (\hspace{1sp}\cite[§4.4]{BEA07}).
\end{remark}

One can see \cite[Section 6]{di09a} that $H_{X_{\mathbb{C}}}$ is isomorphic to one of the following groups:\begin{equation} \{\id\} \ , \ \Z/2\Z \ , \ \Z/3\Z \ , \ \Z/4\Z \ , \ \Z/5\Z \ , \ \Sym_{3} \ , \ \D_{5}.
\label{list_of_possible_group_for_Im_Aut(X_C)} 
\end{equation}
Denote by $\rho$ the restriction of $\rho_{1}$ on the real automorphism group $\Aut_{\mathbb{R}}(X)$, and set $$A_{0} = \text{Ker}(\rho) = G_{X_{\mathbb{C}}} \cap \Aut_{\mathbb{R}}(X), \; A' = \text{Im}(\rho).$$ We will give a description of the families of real automorphisms of $X$, as well as a geometric description of the elements of the groups $A_{0}$ and $A'$ in Section \ref{Sec:subsection_2.3_real_automorphism_groups} below.

\subsection{Real automorphism groups: a geometric description of the families of automorphisms}
\label{Sec:subsection_2.3_real_automorphism_groups}

We proceed by studying each rational real form of quartics del Pezzo $Q_{r,s}(a,b)$, where $(r,s)$ is the signature of a corresponding quadratic form.
We denote by $\sigma$ the antiholomorphic involution $\sigma : (x,y) \mapsto (\overline{y},\overline{x})$ corresponding to $Q_{3,1}$.
We recall that the automorphism group of the quadric $Q_{3,1}$ is given by $\Aut_{\mathbb{R}}(Q_{3,1}) \cong \lbrace (A,\overline{A}) \, \vert \, A \in \PGL_{2}(\mathbb{C}) \rbrace \rtimes \langle \tau : (x,y) \mapsto (y,x) \rangle$ (see Proposition \ref{prop:Proposition_description_Aut(Q31)}).

\subsubsection{The del Pezzo surfaces obtained by blowing up $Q_{3,1}$}
\label{Sec:subsubsection_dP_surf_obtained_by_blowing_up_Q31}

\subsubsection{Case 1: $X \cong Q_{3,1}(0,2)$} Let $X \cong Q_{3,1}(0,2)$ be the blow-up of $Q_{3,1}$ at two pairs of complex conjugate points $(p, \overline{p}), (q, \overline{q})$.

We have sixteen $(-1)$-curves on $X_{\mathbb{C}}$: the exceptional divisors above the blown up points denoted by $E_{p}, E_{\overline{p}}, E_{q}, E_{\overline{q}}$; the strict transforms of the fibres $f, \overline{f}$ passing through one of the four points that we denote by $f_{p}, f_{\overline{p}}, f_{q}, f_{\overline{q}}, \overline{f_{p}}, \overline{f_{\overline{p}}}, \overline{f_{q}}, \overline{f_{\overline{q}}}$; and the strict transforms  of the curves of bidegree $(1,1)$ of $(Q_{3,1})_{\mathbb{C}}$ (linearly equivalent to $f+\overline{f}$) passing through three of the four points and denoted by $ f_{p\overline{p}q}, f_{p\overline{p}\overline{q}}, f_{pq\overline{q}}, f_{\overline{p}q\overline{q}}$.

These $(-1)$-curves form the singular fibres of the ten conic bundle structures on $X_{\mathbb{C}}$ with four singular (complex) fibres each. These conic bundles are given by the linear systems associated to the following classes of divisors:\\
\begin{tabular}{p{6cm}p{7cm}}
{\begin{align}
f+\overline{f}-E_{p}-E_{\overline{p}}  \\  f+\overline{f}-E_{p}-E_{q}  \\  f+\overline{f}-E_{p}-E_{\overline{q}}  \\  f+\overline{f}-E_{q}-E_{\overline{q}}  \\  f+\overline{f}-E_{\overline{p}}-E_{\overline{q}}   
\end{align}}
&%
{\begin{align}
f+\overline{f}-E_{\overline{p}}-E_{q}  \\  f \\  \overline{f}  \\  f+2\overline{f}-E_{p}-E_{\overline{p}}-E_{q}-E_{\overline{q}}  \\  2f+\overline{f}-E_{p}-E_{\overline{p}}-E_{q}-E_{\overline{q}}
\end{align}}
\end{tabular}

The anticanonical divisor class of $X$ is $-K_{X} = 2f+2\overline{f}-E_{p}-E_{\overline{p}}-E_{q}-E_{\overline{q}}$. We now collect these conic bundles in pairs $\lbrace \mathcal{C}_{i} , \mathcal{C}'_{i} \rbrace$, such that $ \mathcal{C}_{i} + \mathcal{C}'_{i} = -K_{X}$ for $i=1,\dots,5$. One has
\begin{align*}
\mathcal{R}_{1} &:= \lbrace \mathcal{C}_{1} , \mathcal{C}'_{1} \rbrace = \lbrace f+\overline{f}-E_{p}-E_{\overline{p}} , f+\overline{f}-E_{q}-E_{\overline{q}} \rbrace,\\
\mathcal{R}_{2} &:= \lbrace \mathcal{C}_{2} , \mathcal{C}'_{2} \rbrace = \lbrace f+\overline{f}-E_{p}-E_{q} , f+\overline{f}-E_{\overline{p}}-E_{\overline{q}} \rbrace,\\
\mathcal{R}_{3} &:= \lbrace \mathcal{C}_{3} , \mathcal{C}'_{3} \rbrace = \lbrace f+\overline{f}-E_{p}-E_{\overline{q}} , f+\overline{f}-E_{\overline{p}}-E_{q} \rbrace,\\
\mathcal{R}_{4} &:= \lbrace \mathcal{C}_{4} , \mathcal{C}'_{4} \rbrace = \lbrace f , f+2\overline{f}-E_{p}-E_{\overline{p}}-E_{q}-E_{\overline{q}} \rbrace,\\
\mathcal{R}_{5} &:= \lbrace \mathcal{C}_{5} , \mathcal{C}'_{5} \rbrace = \lbrace \overline{f} , 2f+\overline{f}-E_{p}-E_{\overline{p}}-E_{q}-E_{\overline{q}} \rbrace.
\end{align*}

We represent with arrows in Figure \ref{Fig:Figure_1_Galois_action_on_conic_bundles_Q31(0,2)} below the way the anti-holomorphic involution $\sigma$ acts on the five pairs of conic bundles.

\begin{figure}[h]
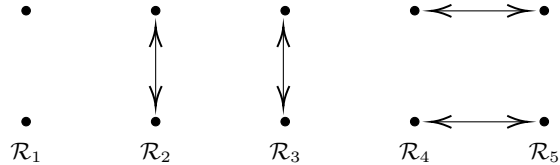

\centering

\caption{Representation of the five pairs of conic bundles and the action of $\sigma$ on them.}
\label{Fig:Figure_1_Galois_action_on_conic_bundles_Q31(0,2)}
\end{figure}

\begin{remark}
As a consequence of the action of the antiholomorphic involution $\sigma$ (Figure \ref{Fig:Figure_1_Galois_action_on_conic_bundles_Q31(0,2)}), we see that $A':=\text{Im}(\rho) \subset \lbrace \id , (23) , (45) , (23)(45) \rbrace$ where the inclusion is strict as one can see from the list \eqref{list_of_possible_group_for_Im_Aut(X_C)}, and that any element $(a_{1},a_{2},a_{3},a_{4},a_{5}) \in \text{Ker}(\rho)=:A_{0}$ has $a_{4}=a_{5}$, so $A_{0}$ embeds into $(\Z/2\Z)^{3}$.
\label{rmk:notice1}
\end{remark}

We recall the following results, obtained in \cite[§4.3]{rob16}.

\begin{lemma}[{\hspace{1sp}\cite[Lemma 4.8]{rob16}}]
Let $p, q \in \p^{1}_{\mathbb{C}} \times \p^{1}_{\mathbb{C}} \cong (Q_{3,1})_{\mathbb{C}}$ be two distinct non conjugate points such that the blow-up of $p, \overline{p}, q, \overline{q}$ is a del Pezzo surface. Then up to automorphisms of $Q_{3,1}$, the points $p$ and $q$ can be chosen to be $p=([1:0],[0:1])$ and $q=([1:1],[1:\mu])$ for some $\mu \in \mathbb{C}\setminus\lbrace 0,\pm1 \rbrace$, respectively.
\label{lem:Lemma_choice_of_points_Q31(0,2)_Rob}
\end{lemma}

\begin{proposition}[{\hspace{1sp}\cite[Propositions 4.9 and 4.10]{rob16}}] 
Let $X \cong Q_{3,1}(0,2)$ be a real del Pezzo surface of degree $4$.
\begin{enumerate}
\item The kernel $A_{0}$ of $\rho : \Aut_{\mathbb{R}}(X) \rightarrow
\Sym_{5}$ is isomorphic to $(\Z/2\Z)^{3}$, and is generated by the elements $\gamma_{1} = (0,1,1,0,0)$, $\gamma_{2} = (1,0,1,0,0)$ and $\gamma_{3} = (0,0,0,1,1)$.
\label{it:item_(1)_Proposition_Aut(X)_Q31(0,2)}
\item The image $A'$ of $\rho$ is either $\langle (23)(45) \rangle$ or trivial. Moreover, the former happens if and only if $ \vert\mu\vert=1$.
\item The equation of the surface $X$ can be given by the intersection of the following two quadrics in $\p^{4}_{x_{1}x_{2}x_{3}x_{4}x_{5}}$
\begin{align*}
Q_{1} &: (\mu-\mu\overline{\mu}+\overline{\mu})x_{1}^{2}-2x_{1}x_{2}+x_{2}^{2}+(1-\overline{\mu}+\mu\overline{\mu}-\mu)x_{3}^{2}+x_{4}^{2}=0,\\
Q_{2} &: \mu\overline{\mu}x_{1}^{2}-2\mu\overline{\mu}x_{1}x_{2}+(\mu-1+\overline{\mu})x_{2}^{2}+\mu\overline{\mu}x_{4}^{2}+(1-\overline{\mu}+\mu\overline{\mu}-\mu)x_{5}^{2}=0.
\end{align*}
\end{enumerate}
\label{Prop:Proposition_Aut(X)_Q31(0,2)}
\end{proposition}

We give an alternative proof of the existence of automorphisms of $X$ corresponding to $\gamma_{1}, \gamma_{2}, \gamma_{3}$ (see point (\ref{it:item_(1)_Proposition_Aut(X)_Q31(0,2)})).

\begin{proof}
As already noticed in Remark \ref{rmk:notice1}, the action of $\sigma$ on the pairs $\mathcal{R}_{4}$ and $\mathcal{R}_{5}$ implies that an automorphism of $X$ which is in the kernel of $\rho$ corresponds to either an element of the form $(\ast,\ast,\ast,0,0)$ or $(*,*,*,1,1)$ in $(\Z/2\Z)^{5}$, which is the same as satisfying the condition $a_{4}+a_{5}=0$. Hence, $A_{0}$ is contained in the group $\lbrace (a_{1},\dots,a_{5}) \in (\Z/2\Z)^{5} \, \vert \, a_{1}+a_{2}+a_{3}=0 \,\, \text{and} \,\, a_{4}+a_{5}=0 \rbrace \cong (\Z/2\Z)^{3}$.

Consider the real automorphism $\alpha_{1} : ([u_{0}:u_{1}],[v_{0}:v_{1}]) \mapsto ([u_{1}:\overline{\mu} u_{0}],[v_{1}:\mu v_{0}])$ of $Q_{3,1}$ which preserves the fibrations $f$, $\overline{f}$ on $(Q_{3,1})_{\mathbb{C}} \cong \p^{1}_{\mathbb{C}} \times \p^{1}_{\mathbb{C}}$, and which interchanges both $p$ with $\overline{p}$ and $q$ with $\overline{q}$. By blowing up $p, \overline{p}, q, \overline{q}$ it lifts to an automorphism of $X$ defined over $\mathbb{R}$ whose action on the Picard group of $X$ with respect to the basis $\lbrace f,\overline{f},E_{p},E_{\overline{p}},E_{q},E_{\overline{q}} \rbrace$ is given by the matrix \[ \small \begin{pmatrix} 1 & 0 & 0 & 0 & 0 & 0 \\ 0 & 1 & 0 & 0 & 0 & 0 \\ 0 & 0 & 0 & 1 & 0 & 0 \\ 0 & 0 & 1 & 0 & 0 & 0 \\ 0 & 0 & 0 & 0 & 0 & 1 \\ 0 & 0 & 0 & 0 & 1 & 0 \end{pmatrix}. \] This automorphism is then in the kernel of $\rho$ and its induced action on the pairs $\mathcal{R}_{i}$'s is given by $\langle \gamma_{1}=(0,1,1,0,0) \rangle$.
The same way, consider the real automorphism $\alpha_{2} : ([u_{0}:u_{1}],[v_{0}:v_{1}]) \mapsto ([u_{1}-\overline{\mu} u_{0}:\overline{\mu}(u_{1}- u_{0})],[v_{1}-\mu v_{0}:\mu(v_{1}-v_{0}])$ of $Q_{3,1}$ which interchanges $p$ with $q$ and $\overline{p}$ with $\overline{q}$, and which preserves the two fibrations $f$, $\overline{f}$. It lifts to an automorphism of $X$ defined over $\mathbb{R}$ contained in the kernel of $\rho$ and whose action on the pairs of conic bundle structures is given by $\langle \gamma_{2}=(1,0,1,0,0) \rangle$. Finally, consider the real birational involution $\phi_{3} : ([u_{0}:u_{1}],[v_{0}:v_{1}]) \dashmapsto ([A_{0}(u_{0},u_{1},v_{0},v_{1}):A_{1}(u_{0},u_{1},v_{0},v_{1})],[B_{0}(u_{0},u_{1},v_{0},v_{1}):B_{1}(u_{0},u_{1},v_{0},v_{1})])$ of $Q_{3,1}$, where
\begin{align*}
A_{0}(u_{0},u_{1},v_{0},v_{1}) &= -v_{1}(\overline{\mu}(\mu-1)u_{0}v_{0} + (1-\overline{\mu})u_{1}v_{1} + (\overline{\mu}-\mu)u_{1}v_{0}),\\
A_{1}(u_{0},u_{1},v_{0},v_{1}) &= -v_{0} \overline{\mu}(\mu(\overline{\mu}-1)u_{0}v_{0} + (\mu-\overline{\mu})u_{0}v_{1} + (1-\mu)u_{1}v_{1}),\\
B_{0}(u_{0},u_{1},v_{0},v_{1}) &= -u_{1}(\mu(\overline{\mu}-1)u_{0}v_{0} + (\mu-\overline{\mu})u_{0}v_{1} + (1-\mu)u_{1}v_{1}),\\
B_{1}(u_{0},u_{1},v_{0},v_{1}) &= -u_{0}\mu(\overline{\mu}(\mu-1)u_{0}v_{0} + (1-\overline{\mu})u_{1}v_{1} + (\overline{\mu}-\mu)u_{1}v_{0}),
\end{align*}
and whose base-points are $p$, $\overline{p}$, $q$ and $\overline{q}$. By blowing up $p, \overline{p}, q, \overline{q}$ it lifts to a biregular morphism of $X$, defined over $\mathbb{R}$, whose action on the Picard group of $X$ is given by the matrix \[ \small \begin{pmatrix} 1 & 2 & 1 & 1 & 1 & 1 \\ 2 & 1 & 1 & 1 & 1 & 1 \\ -1 & -1 & 0 & -1 & -1 & -1 \\ -1 & -1 & -1 & 0 & -1 & -1 \\ -1 & -1 & -1 & -1 & 0 & -1 \\ -1 & -1 & -1 & -1 & -1 & 0 \end{pmatrix}. \] It follows that this automorphism is in the kernel of $\rho$ and its action on the pairs $\mathcal{R}_{i}$'s is given by $\langle \gamma_{3}=(0,0,0,1,1) \rangle$. This yields the claim of point \ref{it:item_(1)_Proposition_Aut(X)_Q31(0,2)}.
\end{proof}

\begin{remark} From the description of families of automorphisms in $\Aut_{\mathbb{R}}(X)$ given in Proposition \ref{Prop:Proposition_Aut(X)_Q31(0,2)} one can see that for the general quartic del Pezzo surfaces $X$ of type $Q_{3,1}(0,2)$, the action of $\Gal(\mathbb{C}/\mathbb{R})$ on $\Pic(X)$ (see Figure \ref{Fig:Figure_1_Galois_action_on_conic_bundles_Q31(0,2)}) is not realized by an automorphism of the surface. 
\end{remark}

\subsubsection{Case 2: $X \cong Q_{3,1}(2,1)$} Now let $X \cong \p^{2}_{\mathbb{R}}(1,2)$ be the blow-up of $\p^{2}_{\mathbb{R}}$ at one real point $q$ and two pairs of complex conjugate points $S=\lbrace s,\overline{s} \rbrace$ and $R=\lbrace r, \overline{r} \rbrace$. By first blowing up $\p^{2}_{\mathbb{R}}$ in $S$ and then blowing down the strict transform of the real line passing through $s$ and $\overline{s}$, we note that $X$ is isomorphic to the blow-up of the quadric $Q_{3,1}$ at two rational points $p$, $q$ and one pair of complex conjugate points $R=\lbrace r, \overline{r} \rbrace$, that we denote by $Q_{3,1}(2,1)$. In what follows, we will work with this second birational model of $X$.

In this case, using the same notations as before, one has \begin{align*}
\mathcal{R}_{1} &:= \lbrace f+\overline{f}-E_{p}-E_{q} , f+\overline{f}-E_{r}-E_{\overline{r}} \rbrace,\\
\mathcal{R}_{2} &:= \lbrace f+\overline{f}-E_{p}-E_{r} , f+\overline{f}-E_{q}-E_{\overline{r}} \rbrace,\\
\mathcal{R}_{3} &:=\lbrace f+\overline{f}-E_{p}-E_{\overline{r}} , f+\overline{f}-E_{q}-E_{r} \rbrace,\\
\mathcal{R}_{4} &:= \lbrace f , f+2\overline{f}-E_{p}-E_{q}-E_{r}-E_{\overline{r}} \rbrace,\\
\mathcal{R}_{5} &:= \lbrace \overline{f} , 2f+\overline{f}-E_{p}-E_{q}-E_{r}-E_{\overline{r}} \rbrace.
\end{align*}
The action of the complex involution on these pairs of conic bundles is then represented in Figure \ref{Fig:Figure_2_Galois_action_on_conic_bundles_Q31(2,1)}.

\begin{figure}
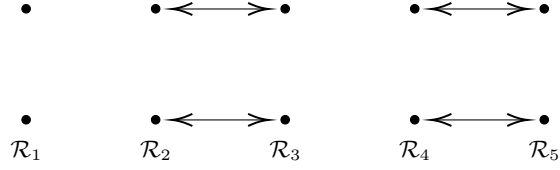

\centering
 
\caption{The action of $\sigma$ on the five pairs of conic bundles.}
\label{Fig:Figure_2_Galois_action_on_conic_bundles_Q31(2,1)}
\end{figure}

As in the previous case, this gives immediate restrictions on the groups $A_{0}$ and $A'$, as it is shown in \cite[Proposition 6.6]{yas22}; namely, $A_{0}$ is isomorphic to a subgroup of $\lbrace (0,0,0,0,0),(0,0,0,$ $1,1),(0,1,1,0,0),\\(0,1,1,1,1)\rbrace \cong (\Z/2\Z)^{2}$, and $A'$ lies in the group $\langle a=(23), b=(2435) \, \vert \, a^{2}=b^{4}=1, aba=b^{-1} \rangle \cong \D_{4}$.

We now go a bit further in describing geometrically the elements of the groups $A_{0}$ and $A'$.

\begin{lemma}
Let $p, q, r, \overline{r} \in \p^{1}_{\mathbb{C}} \times \p^{1}_{\mathbb{C}} \cong (Q_{3,1})_{\mathbb{C}}$ be respectively two real points and two complex conjugate points such that the blow-up of $p, q, r, \overline{r}$ is a real del Pezzo surface. Then up to automorphisms of $Q_{3,1}$, the points $p$, $q$ and $r$ can be chosen to be $p=([1:0],[1:0])$, $q=([0:1],[0:1])$ and $r=([1:1],[\mu:1])$ for some $\mu \in \mathbb{C}\setminus\lbrace 0,\pm1 \rbrace$, respectively.
\label{lem:Lemma_choice_of_points_Q31(2,1)}
\end{lemma}

\begin{proof}
We have $p=([a:b],[\overline{a}:\overline{b}])$ and $q=([c:d],[\overline{c}:\overline{d}])$ for some $a, b, c, d \in \mathbb{C}$, and $ad-bc \neq 0$ because $p$ and $q$ are not on the same ruling of $(Q_{3,1})_{\mathbb{C}}$. It follows that the map $A : [u:v] \mapsto [du-cv:-bu+av]$ is contained in $\PGL_{2}(\mathbb{C})$. Then $(A,\overline{A}) \in \Aut_{\mathbb{R}}(Q_{3,1})$ sends respectively $p$ and $q$ onto $([1:0],[1:0])$ and $([0:1],[0:1])$. Now we may assume that $r=([\lambda:1],[\rho:1])$ with $\lambda, \rho \in \mathbb{C}^{*}$ because by hypothesis the points are not on the same fibres by any projection.
We take $$ (B,\overline{B}) = \left( \begin{pmatrix} \frac{1}{\lambda} & 0 \\[1mm] 0 & 1\end{pmatrix} , \begin{pmatrix} \frac{1}{\overline{\lambda}} & 0 \\[1mm] 0 & 1\end{pmatrix} \right) \in \Aut_{\mathbb{R}}(Q_{3,1}),$$ it fixes $p$ and $q$, and it sends $r$ and $\overline{r}$ onto $([1:1],[\mu:1])$ and $([\overline{\mu}:1],[1:1])$, respectively, where $\mu=\frac{\rho}{\overline{\lambda}}$.

If $\mu \in \lbrace 0,\pm1 \rbrace$, $X$ is not a del Pezzo surface. Indeed, when $\mu=1$ the points $r$ and $\overline{r}$ are equal; when $\mu=0$ the points $q$ and $\overline{r}$ are on the same fibre, as well as the points $q$ and $r$; and finally, when $\mu=-1$ there is a real curve of bidegree $(1,1)$ of $(Q_{3,1})_{\mathbb{C}}$ passing through the four points.
\end{proof}

\begin{proposition} Let $X \cong Q_{3,1}(2,1)$ be a real del Pezzo surface of degree $4$.
\begin{enumerate}
\item The kernel $A_{0}$ of $\rho : \Aut_{\mathbb{R}}(X) \rightarrow
\Sym_{5}$ is isomorphic to $(\Z/2\Z)^{2}$, and it is generated by the elements $\gamma_{1} = (0,1,1,0,0)$ and $\gamma = (0,0,0,1,1)$.
\label{it:item_(1)_Proposition_Aut(X)_Q31(2,1)}
\item The image $A'$ of $\rho$ is either $\langle (23)(45) \rangle$ or $\lbrace \id \rbrace$. Moreover, the former happens if and only if $ \mu \in \mathbb{R} \setminus \lbrace 0,\pm 1 \rbrace$.
\label{it:item_(2)_Proposition_Aut(X)_Q31(2,1)}
\end{enumerate}
Furthermore, the elements $\gamma_{1}$ and $\gamma$ are realized in $A_{0}$ as the lift of an involution of $Q_{3,1}$ and the lift of a birational involution of $Q_{3,1}$ with base-points $p, q, r, \overline{r}$, respectively. The element $\tau=(23)(45)$ can be realized as the lift on $X$ of an involution of $Q_{3,1}$.
\label{Prop:Proposition_Aut(X)_Q31(2,1)}
\end{proposition}

\begin{proof}
(\ref{it:item_(1)_Proposition_Aut(X)_Q31(2,1)}) We have to show the existence of automorphisms of $X$ corresponding to the elements $\gamma_{1}$ and $\gamma$ in $(\Z/2\Z)^{5}$.

Consider the real involution $ \delta_{1} : ([u_{0}:u_{1}],[v_{0}:v_{1}]) \mapsto ([\overline{\mu}u_{1}:u_{0}],[\mu v_{1}:v_{0}])$ of $Q_{3,1}$ which preserves the fibrations $f, \overline{f}$ on $(Q_{3,1})_{\mathbb{C}}$, and which interchanges both $p$ with $q$ and $r$ with $\overline{r}$. By blowing up $p, q, r, \overline{r}$ it lifts to an automorphism of $X$ defined over $\mathbb{R}$ which is contained in the kernel of $\rho$ and whose induced action on the exceptional pairs $\mathcal{R}_{i}$'s is given by $\langle \gamma_{1}=(0,1,1,0,0) \rangle$. Now consider the real birational involution $\phi : ([u_{0}:u_{1}],[v_{0}:v_{1}]) \dashmapsto ([A_{0}(u_{0},u_{1},v_{0},v_{1}):A_{1}(u_{0},u_{1},v_{0},v_{1})],[B_{0}(u_{0},u_{1},v_{0},v_{1}):B_{1}(u_{0},u_{1},v_{0},v_{1})])$ of $Q_{3,1}$, where
\begin{align*}
A_{0}(u_{0},u_{1},v_{0},v_{1}) &= \overline{\mu}v_{0}((\overline{\mu}\mu-1)u_{1}v_{1} + (1-\mu)u_{0}v_{1} + (1-\overline{\mu})u_{1}v_{0}),\\
A_{1}(u_{0},u_{1},v_{0},v_{1}) &= v_{1}(\mu(\overline{\mu}-1)u_{0}v_{1} + (1-\mu \overline{\mu})u_{0}v_{0} + \overline{\mu}(\mu-1)u_{1}v_{0}),\\
B_{0}(u_{0},u_{1},v_{0},v_{1}) &= -\mu u_{0}((\mu \overline{\mu}-1)u_{1}v_{1} + (1-\mu)u_{0}v_{1} + (1-\overline{\mu})u_{1}v_{0}),\\
B_{1}(u_{0},u_{1},v_{0},v_{1}) &= -u_{1}(\mu(\overline{\mu}-1)u_{0}v_{1} + (1-\mu \overline{\mu})u_{0}v_{0} + \overline{\mu}(\mu-1)u_{1}v_{0}),
\end{align*}
and whose base-points are $p$, $q$, $r$ and $\overline{r}$. By blowing up these points it lifts to a biregular morphism of $X$, defined over $\mathbb{R}$, and whose action on the Picard group of $X$ is given by the matrix\[ \small \begin{pmatrix} 1 & 2 & 1 & 1 & 1 & 1 \\ 2 & 1 & 1 & 1 & 1 & 1 \\ -1 & -1 & 0 & -1 & -1 & -1 \\ -1 & -1 & -1 & 0 & -1 & -1 \\ -1 & -1 & -1 & -1 & 0 & -1 \\ -1 & -1 & -1 & -1 & -1 & 0 \end{pmatrix},\]with respect to the basis $\lbrace f,\overline{f},E_{p},E_{q},E_{r},E_{\overline{r}} \rbrace$. It follows that this automorphism is in the kernel of $\rho$ and its action on the pairs of conic bundle structures is the one of the element $\gamma=(0,0,0,1,1)$.

(\ref{it:item_(2)_Proposition_Aut(X)_Q31(2,1)}) Recall that $A'$ is a subgroup of $\langle a=(23), b=(2435) \, \vert \, a^{2}=b^{4}=1, aba=b^{-1} \rangle \cong \D_{4}$ isomorphic to $\lbrace \id \rbrace$, $\Z/2\Z$ or $\Z/4Z$ (see list \eqref{list_of_possible_group_for_Im_Aut(X_C)}).

If there is an automorphism $\alpha$ of $X$ exchanging $\mathcal{R}_{2}$ and $\mathcal{R}_{3}$, then its action on those pairs would be given either by \[(1) \, \myinline{
}.\]
We may assume that the action is as in (1) since we can compose an automorphism whose action is as in (2) with an element of $A_{0}$ that corresponds to $\gamma_{1}=(0,1,1,0,0)$ to obtain an automorphism whose action is that in (1). On the pairs $\mathcal{R}_{4}$ and $\mathcal{R}_{5}$ the action of $\alpha$ is then given either by \[(3) \,  \myinline{%
};\] and as before, we may assume that its action is as in (3) by composing an automorphism whose action is as in (4) with an element of $A_{0}$ corresponding to $\gamma=(0,0,0,1,1)$ to obtain an automorphism whose action is that in (3). Summarising, we have to study two cases:
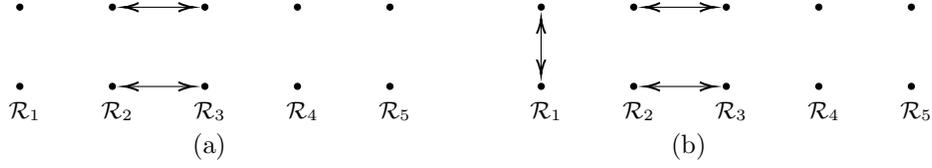
\begin{center} %
 \end{center}
In both cases (a) and (b), $f, \overline{f}$ are preserved by looking at pairs $\mathcal{R}_{4}, \mathcal{R}_{5}$, and hence $f+\overline{f}$ is preserved. In the case (a), looking at the pair $\mathcal{R}_{1}$ we see that $f+\overline{f}-E_{p}-E_{q}, f+\overline{f}-E_{r}-E_{\overline{r}}$ are preserved; then $E_{p}+E_{q}$ and $E_{r}+E_{\overline{r}}$ are preserved while the action on pairs $\mathcal{R}_{2}$ and $\mathcal{R}_{3}$ gives that $\alpha$ interchanges $E_{p}+E_{r}$ with $E_{p}+E_{\overline{r}}$ and $E_{q}+E_{\overline{r}}$ with $E_{q}+E_{r}$. This implies that $E_{p}, E_{q}$ are fixed and $E_{r}, E_{\overline{r}}$ are exchanged. 
So $\alpha$ comes from an automorphism $\alpha'$ of $(Q_{3,1})_{\mathbb{C}} \cong \p^{1}_{\mathbb{C}} \times \p^{1}_{\mathbb{C}}$ defined over $\mathbb{R}$ which preserves the fibrations, fixes $p, q$ and interchanges $r$ with $\overline{r}$. Let us see that such an automorphism does not exist. Indeed, the automorphism $\alpha'$ would be given by $(x,y) \mapsto (Ax,\overline{A}y)$ where $A \in \PGL_{2}(\mathbb{C})$ (\cite[Proposition 4.1]{rob16}) with $\alpha'(p)=p$, $\alpha'(q)=q$,  and so $A=\begin{bmatrix} \lambda & 0 \\ 0 & 1 \\ \end{bmatrix}$ is diagonal with $\lambda \in \mathbb{C}^{*}$ under the choice of the points $p=([1:0],[1:0])$ and $q=([0:1],[0:1])$. Since $\alpha'(r)=\overline{r}$, where $r=([1:1],[\mu:1])$ for some $\mu \in \mathbb{C}\setminus\lbrace0,\pm1\rbrace$ (Lemma \ref{lem:Lemma_choice_of_points_Q31(2,1)}), we have $\lambda=\overline{\mu}$ and $\overline{\lambda}\mu=1$ and hence $\mu^{2}=1$, which gives a contradiction. So an automorphism of $X$ of type (a) does not exist. In case (b), $\alpha$ is not even an automorphism of the Picard group because the matrix corresponding to an action described in (b) with respect to the basis $\lbrace f,\overline{f},E_{p},E_{q},E_{r},E_{\overline{r}} \rbrace$ is \[ 
 \notin \GL_{6}(\mathbb{Z}). \] Therefore, an automorphism that acts as $(23)$ does not belong to the image of $\rho$.

Now, we check that there is an automorphism of $X$ which acts as $(23)(45)$ if and only if $\mu \in \mathbb{R}$. As before, we can see that automorphisms corresponding to $(23)(45)$ are, up to composition with an element of $A_{0}$ corresponding to $\gamma_{1}+\gamma=(0,1,1,1,1)$, of the form \begin{center} %
 \end{center}
For the case (a), looking at the pairs $\mathcal{R}_{4}$ and $\mathcal{R}_{5}$ we see that $f$ and $\overline{f}$ are exchanged, and then $f+\overline{f}$ is preserved. The exchange of pairs $\mathcal{R}_{2}$ and $\mathcal{R}_{3}$ gives that $f+\overline{f}-E_{p}-E_{r}$ with $f+\overline{f}-E_{p}-E_{\overline{r}}$ are interchanged and so are $f+\overline{f}-E_{q}-E_{\overline{r}}$ with $f+\overline{f}-E_{q}-E_{r}$. This implies that $E_{p}+E_{r}$ with $E_{p}+E_{\overline{r}}$ are interchanged and $E_{q}+E_{\overline{r}}$ with $E_{q}+E_{r}$ are interchanged, respectively. So an automorphism of type $(23)(45)$ for case (a) comes from an automorphism $\delta$ of $(Q_{3,1})_{\mathbb{C}} \cong \p^{1}_{\mathbb{C}} \times \p^{1}_{\mathbb{C}}$ defined over $\mathbb{R}$ which interchanges $f$ with $\overline{f}$, $r$ with $\overline{r}$, and fixes $p$ and $q$. So $\delta$ is given by $\delta : (x,y) \mapsto (\overline{A}y,Ax)$ where $A \in \PGL_{2}(\mathbb{C})$ is satisfying $$\overline{A} 
.$$ 
This implies that $$A = %
.$$ Hence, $\lambda=1$ and $\mu = \overline{\mu}$. Therefore this automorphism exists if and only if $\mu \in \mathbb{R}$. Case (b) is not possible because the matrix of its action on the Picard group with respect to the basis $\lbrace f,\overline{f},E_{p},E_{q},E_{r},E_{\overline{r}} \rbrace$ is \small\[ %
 \notin \GL_{6}(\mathbb{Z}), \]and this shows that it is not even an automorphism of the Picard group.

Suppose that there is an automorphism of $X$ acting as the double transposition $(25)(34)$ on the five pairs of conic bundle structures. We may assume that its action is one of the following two by composing with an automorphism of $A_{0}$ corresponding to $\gamma_{1}=(0,1,1,0,0)$ or to $\gamma=(0,0,0,1,1)$: \begin{center} %
 \end{center}
In case (a), the induced action on the Picard group of $X$ with respect to the basis $\{f,\overline{f},E_{p},E_{q},E_{r},E_{\overline{r}}\}$ is given by the matrix \[ \small\ %
,\] meaning that such an automorphism, if it exists, comes from the lift of a real birational involution $\psi$ of $Q_{3,1}$ with base-points $p=([1:0],[1:0]), r=([1:1],[\mu:1]), \overline{r}=([\overline{\mu}:1],[1:1])$, and which fixes $q=([0:1],[0:1])$; it contracts the curves $f_{p}$ onto $r$, $\overline{f_{p}}$ onto $\overline{r}$ and $f_{pr\overline{r}}$ onto $p$. We can check by explicit computations that such a birational involution of bidegree $(1,1)$ exists if and only if $\mu+\overline{\mu}=2$ and is of the form $\psi : ([u_{0}:u_{1}],[v_{0}:v_{1}]) \dashmapsto ([u_{0}v_{1}-\overline{\mu}u_{1}v_{0}:u_{0}v_{1}-u_{1}v_{0}+(\mu-1)u_{1}v_{1}],[\mu u_{0}v_{1}-u_{1}v_{0}:u_{0}v_{1}-u_{1}v_{0}+(\mu-1)u_{1}v_{1}])$. 
Case (b) is not possible because the matrix of an action described in (b) on the Picard group with respect to the basis $\{f,\overline{f},E_{p},E_{q},E_{r},E_{\overline{r}}\}$ is given by \small\[ 
 \notin \GL_{6}(\mathbb{Z}), \] which shows that it is not even an automorphism of the Picard group.

Therefore, an automorphism of type $(24)(35)$ would exist if and only if automorphisms of type $(23)(45)$ and $(25)(34)$ exist, that is if and only if $\mu+\overline{\mu}=2$ and $\mu=\overline{\mu}$ which implies $\mu=1$, which is not possible. So, an automorphism that acts as $(24)(35)$ does not belong to the image of $\rho$.
   
Finally, the action of automorphisms of type $(2435)$ on the pairs $\mathcal{R}_{i}$'s of conic bundles is of the form \begin{center} %
 \end{center}
by composing with the element of $A_{0}$ corresponding to $\gamma_{1}+\gamma=(0,1,1,1,1)$. In case (a), the induced action of such an automorphism on the Picard group of $X$ with respect to the basis $\{f,\overline{f},E_{p},E_{q},E_{r},E_{\overline{r}}\}$ is given by the matrix \[ \small\ %
.\] If such an automorphism exists, it comes from a real birational map $\psi$ of $Q_{3,1}$ of order $4$ with $3$ base-points, $p=[1:0],[1:0]$, $r=[1:1],[\mu:1]$ and $\overline{r}=[\overline{\mu}:1],[1:1]$, and which fixes the point $q=[0:1],[0:1]$. It is such that $\psi^{2}$ induces an automorphism of $X$ of type $(23)(45)$, which exists if and only if $\mu \in \mathbb{R}$ and which is coming from the involution $\delta : ([u_{0}:u_{1}],[v_{0}:v_{1}]) \mapsto ([v_{0}:v_{1}],[u_{0}:u_{1}])$ of $Q_{3,1}$. Since the rational inverse of $\psi$ contracts $f_{p r \overline{r}}$ onto $p$, $\overline{f_{p}}$ onto $r$ and $f_{p}$ onto $\overline{r}$, one deduces that $\psi$ contracts $f_{p}$ onto $r$, $\overline{f_{p}}$ onto $\overline{r}$ and $f_{p r \overline{r}}$ onto $p$. One can check by explicit computations that such a birational transformation of $\p^{1}_{\mathbb{C}} \times \p^{1}_{\mathbb{C}} \cong (Q_{3,1})_{\mathbb{C}}$ exists only if $\mu=1$, but this value of $\mu$ is prohibited (see Lemma \ref{lem:Lemma_choice_of_points_Q31(2,1)}). Case (b) is not even an automorphism of the Picard group since the matrix corresponding to the action described in (b) with respect to the basis $\{f,\overline{f},E_{p},E_{q},E_{r},E_{\overline{r}}\}$ is \small\[ \begin{pmatrix} 1 & 1 & 1 & 0 & 0 & 1 \\ 1 & 1 & 1 & 0 & 1 & 0 \\ -1 & -1 & -\frac{3}{2} & -\frac{1}{2} & -\frac{1}{2} & -\frac{1}{2} \\ 0 & 0 & -\frac{1}{2} & \frac{1}{2} & \frac{1}{2} & \frac{1}{2} \\ 0 & -1 & -\frac{1}{2} & \frac{1}{2} & -\frac{1}{2} & -\frac{1}{2} \\ -1 & 0 & -\frac{1}{2} & \frac{1}{2} & -\frac{1}{2} & -\frac{1}{2} \end{pmatrix} \notin \GL_{6}(\mathbb{Z}). \] Therefore, an element of type $(2435)$ does not belong to the image of $\rho$.
 
\end{proof}  

\begin{remark} As in the case of $Q_{3,1}(0,2)$, one can see from the description of families of automorphisms given in Proposition \ref{Prop:Proposition_Aut(X)_Q31(2,1)} that for the general quartic del Pezzo surfaces of type $Q_{3,1}(2,1)$, the action of $\Gal(\mathbb{C}/\mathbb{R})$ on the Picard group (see Figure \ref{Fig:Figure_2_Galois_action_on_conic_bundles_Q31(2,1)}) is not realized by an automorphism of the surface. 
\end{remark}

\subsubsection{Case 3: $X \cong Q_{3,1}(4,0)$} Let $X \cong \p^{2}_{\mathbb{R}}(3,1) \cong Q_{3,1}(4,0)$ be the real rational quartic del Pezzo obtained by blowing up $\p^{2}_{\mathbb{R}}$ at three real points $p,q,r$ and one pair of complex conjugate points $S=\{s,\overline{s}\}$. Since we can first blow up $\p^{2}_{\mathbb{R}}$ in $S$ and then blowing down the strict transform of the real line of $\p^{2}_{\mathbb{R}}$ passing through $s$ and $\overline{s}$ onto the quadric $Q_{3,1}$, we see that $X$ is isomorphic to the blow-up of $Q_{3,1}$ at four general real points, say $p, q, r$ and $s$. From now on, we will work with this second description of the surface $X$.

In this case, and using the same notations as before, we get \begin{align*}
\mathcal{R}_{1} &:= \lbrace f+\overline{f}-E_{p}-E_{q} , f+\overline{f}-E_{r}-E_{s} \rbrace,\\
\mathcal{R}_{2} &:= \lbrace f+\overline{f}-E_{p}-E_{r} , f+\overline{f}-E_{q}-E_{s} \rbrace,\\
\mathcal{R}_{3} &:=\lbrace f+\overline{f}-E_{p}-E_{s} , f+\overline{f}-E_{q}-E_{r} \rbrace,\\
\mathcal{R}_{4} &:= \lbrace f , f+2\overline{f}-E_{p}-E_{q}-E_{r}-E_{s} \rbrace,\\
\mathcal{R}_{5} &:= \lbrace \overline{f} , 2f+\overline{f}-E_{p}-E_{q}-E_{r}-E_{s} \rbrace,
\end{align*}
and we represent in Figure \ref{Fig:Figure_3_Galois_action_on_conic_bundles_Q31(4,0)} below the way the anti-holomorphic involution $\sigma$ acts on the above five pairs of conic bundles.
\begin{figure}[h]
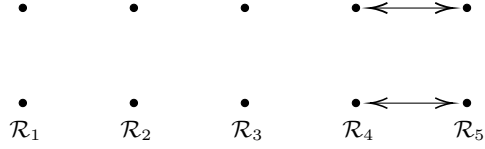

\centering

\caption{The action of $\sigma$ on the pairs of conic bundle structures.}
\label{Fig:Figure_3_Galois_action_on_conic_bundles_Q31(4,0)}
\end{figure} 
As an immediate consequence of the given Galois action (Figure \ref{Fig:Figure_3_Galois_action_on_conic_bundles_Q31(4,0)}), we see that for any element $(a_{1},\dots,a_{5})$ in $A_{0}$ one has $a_{4}=a_{5}$ so that $A_{0}$ lies inside the group $\lbrace (a_{1},a_{2},a_{3},a_{4},a_{5}) \in (\Z/2\Z)^{5} \, \vert \, a_{4}+a_{5}=a_{1}+a_{2}+a_{3}=0 \rbrace \cong (\Z/2\Z)^{3}$, and $A'$ is a subgroup of $\Sym\{\mathcal{R}_{1},\mathcal{R}_{2},\mathcal{R}_{3}\} \times \Sym\{\mathcal{R}_{4},\mathcal{R}_{5}\} \cong \Sym_{3} \times \Z/2\Z \cong \D_{6}$ (see \cite[Proposition 6.7]{yas22}) isomorphic to $\{\id\}, \Z/2\Z, \Z/3\Z$ or $\Sym_{3}$ as one can see from list \eqref{list_of_possible_group_for_Im_Aut(X_C)}.

Let us go a bit further, studying the geometry of elements of the groups $A_{0}$ and $A'$.

\begin{lemma}
Let $p, q, r, s \in \p^{1}_{\mathbb{C}} \times \p^{1}_{\mathbb{C}} \cong (Q_{3,1})_{\mathbb{C}}$ be four real points such that the blow-up of $Q_{3,1}$ in $p, q, r, s$ is a real del Pezzo surface. Then, up to automorphisms of $Q_{3,1}$, the points can be chosen to be $p=([1:0],[1:0]), q=([0:1],[0:1]), r=([1:1],[1:1])$ and $s=([\lambda:1],[\overline{\lambda}:1])$ for some $\lambda \in \mathbb{C}\setminus\mathbb{R}$.
\label{lem:Lemma_choice_of_points_Q31(4,0)}  
\end{lemma}

\begin{proof}
We have $p=([a:b],[\overline{a}:\overline{b}]), q=([c:d],[\overline{c}:\overline{d}]), r=([e:f],[\overline{e}:\overline{f}])$ and $s=([g:h],[\overline{g}:\overline{h}])$ for some $a, b, c, d, e, f, g, h \in \mathbb{C}$, and $ad-bc \neq 0$, $ah-bg \neq 0$, $ch-dg \neq 0$ because $p, q, r, s$ are pairwise not on the same fibration of $\p^{1}_{\mathbb{C}} \times \p^{1}_{\mathbb{C}}$ by any projection. Then, thanks to \cite[Lemma 2.2.3]{boi25}, 
there exists an automorphism $\alpha \in \Aut_{\mathbb{R}}(Q_{3,1})$ such that $\alpha(p)=([1:0],[1:0]), \alpha(q)=([0:1],[0:1])$ and $\alpha(r)=([1:1],[1:1])$, and we may assume that $\alpha(s)=([\lambda:1],[\overline{\lambda}:1])$ with $\lambda \in \mathbb{C} \setminus \mathbb{R}$ by the hypothesis made on the four points.

Notice that if $\lambda=0$, the points $q$ and $s$ are on the same fibre, and so are the points $r$ and $s$ if $\lambda=1$; and if $\lambda \in \mathbb{R}\setminus\{0,1\}$ there is a real curve of bidegree $(1,1)$ passing through the four points whose equation is given by $\{u_{0}v_{1}-u_{1}v_{0}=0\}$ in $\p^{1}_{\mathbb{C},u_{0}u_{1}} \times \p^{1}_{\mathbb{C},v_{0}v_{1}}$.
\end{proof}

\begin{proposition}
Let $X \cong Q_{3,1}(4,0)$ be a real quartic del Pezzo surface.
\begin{enumerate}
\item The kernel $A_{0}$ of the morphism $ \rho : \Aut_{\mathbb{R}}(X) \rightarrow \Sym_{5} $ is isomorphic to $(\Z/2\Z)^{3}$, and it is generated by the elements $\gamma_{1}=(0,1,1,0,0), \gamma_{2}=(1,0,1,0,0)$ and $\gamma_{3}=(0,0,0,1,1)$.
\label{it:item_(1)_Proposition_Aut(X)_Q31(4,0)}
\item The image $A'$ of the morphism $\rho$ is either $\langle (123),(12)(45) \rangle$ or $\langle (12)(45) \rangle$ or $\lbrace \id \rbrace$. Moreover, $A'=\langle (123),(12)(45) \rangle \cong \Sym_{3}$ if and only if $\lambda=e^{\pm i \frac{\pi}{3}}$, $A'=\langle (12)(45) \rangle \cong \Z/2\Z$ if and only if $\lambda + \overline{\lambda}=1, \lambda \neq e^{\pm i \frac{\pi}{3}}$, and $A'=\lbrace \id \rbrace$ otherwise.
\label{it:item_(2)_Proposition_Aut(X)_Q31(4,0)}
\end{enumerate}
Furthermore, the elements $\gamma_{1}$ and $\gamma_{2}$ are realized in $A_{0}$ as the lifts of involutions of $Q_{3,1}$, and $\gamma_{3}$ as the lift of a birational involution of $Q_{3,1}$ with base-points $p, q, r, s$. The elements $\tau_{1}=(12)(45)$ and $\tau_{2}=(123)$ can be realized as the lifts on $X$ of an involution and an automorphism of order three of $Q_{3,1}$, respectively.
\label{Prop:Proposition_Aut(X)_Q31(4,0)}
\end{proposition}

\begin{proof}
(\ref{it:item_(1)_Proposition_Aut(X)_Q31(4,0)}) We have to prove the existence of automorphisms of $X$ corresponding to the elements $\gamma_{1}, \gamma_{2}$ and $\gamma_{3}$ in $(\Z/2\Z)^{5}$.\\
Consider the real involution $\alpha_{1} : ([u_{0}:u_{1}],[v_{0}:v_{1}]) \mapsto ([\lambda u_{1}:u_{0}],[\overline{\lambda} v_{1}:v_{0}]) \in \Aut_{\mathbb{R}}(Q_{3,1})$ which preserves the fibrations $f, \overline{f}$ on $(Q_{3,1})_{\mathbb{C}} \cong \p^{1}_{\mathbb{C}} \times \p^{1}_{\mathbb{C}}$, and which interchanges both $p$ with $q$ and $r$ with $s$. It lifts by blowing up the points $p, q, r, s$ to an automorphism $\widehat{\alpha_{1}}$ of $X$ defined over $\mathbb{R}$ that is contained in the kernel of $\rho$ and whose action on the pairs $\mathcal{R}_{i}$'s is given by $\langle \gamma_{1}=(0,1,1,0,0) \rangle$. In the same way, consider the real automorphism $\alpha_{2} : ([u_{0}:u_{1}],[v_{0}:v_{1}]) \mapsto ([u_{0}-\lambda u_{1}:u_{0}-u_{1}],[v_{0}-\overline{\lambda} v_{1}:v_{0}-v_{1}]) \in \Aut_{\mathbb{R}}(Q_{3,1})$ which interchanges $p$ with $r$ and $q$ with $s$, and which preserves the two fibrations $f, \overline{f}$. It is an involution that lifts to an automorphism $\widehat{\alpha_{2}}$ of $X$ defined over $\mathbb{R}$ which is in the kernel of $\rho$ and whose action on the pairs of conic bundles is given by $\langle \gamma_{2}=(1,0,1,0,0) \rangle$.
Finally, consider the real birational involution $\phi_{3} : ([u_{0}:u_{1}],[v_{0}:v_{1}]) \dashmapsto ([A_{0}(u_{0},u_{1},v_{0},v_{1}):A_{1}(u_{0},u_{1},v_{0},v_{1})],[B_{0}(u_{0},u_{1},v_{0},v_{1}):B_{1}(u_{0},u_{1},v_{0},v_{1})])$ of $Q_{3,1}$, with
\begin{align*}
A_{0}(u_{0},u_{1},v_{0},v_{1}) &= \lambda v_{0}((\overline{\lambda}-1)u_{0}v_{1} + (\lambda - \overline{\lambda})u_{1}v_{1} + (1-\lambda)u_{1}v_{0}),\\
A_{1}(u_{0},u_{1},v_{0},v_{1}) &= v_{1}(\overline{\lambda}(\lambda-1)u_{0}v_{1} + \lambda(1-\overline{\lambda})u_{1}v_{0} + (\overline{\lambda}-\lambda)u_{0}v_{0}),\\
B_{0}(u_{0},u_{1},v_{0},v_{1}) &= -\overline{\lambda} u_{0}((\overline{\lambda}-1)u_{0}v_{1} + (\lambda-\overline{\lambda})u_{1}v_{1} + (1-\lambda)u_{1}v_{0}),\\
B_{1}(u_{0},u_{1},v_{0},v_{1}) &= -u_{1}(\overline{\lambda}(\lambda-1)u_{0}v_{1} + \lambda(1-\overline{\lambda})u_{1}v_{0} + (\overline{\lambda}-\lambda)u_{0}v_{0}),
\end{align*}
and whose base-points are $p, q, r$ and $s$. It lifts, by blowing up these four real points, to an automorphism $\widehat{\phi_{3}}$ of $X$, which is in the kernel of $\rho$ and whose action on the pairs $\mathcal{R}_{i}$'s corresponds to $\gamma_{3}=(0,0,0,1,1)$. This yields the claim of point \ref{it:item_(1)_Proposition_Aut(X)_Q31(4,0)}.

(\ref{it:item_(2)_Proposition_Aut(X)_Q31(4,0)}) Recall that the image $A'$ is a subgroup of $\langle (123),(12) \rangle \times \langle (45) \rangle \cong \D_{6}$ isomorphic to $\lbrace \id \rbrace, \Z/2\Z, \Z/3\Z$ or $\Sym_{3}$ (see list \eqref{list_of_possible_group_for_Im_Aut(X_C)}). We show that the elements $(12)$ and $(45)$ do not belong to the image, while $(123)$ does if and only if $\lambda=e^{\pm i \frac{\pi}{3}}$ and $(12)(45)$ does if and only if $\lambda+\overline{\lambda}=1$.

We start explaining why there is no automorphism of type $(12)$. If there were an automorphism $\alpha$ of $X$ exchanging the pair $\mathcal{R}_{1}$ with $\mathcal{R}_{2}$ then the action of $\alpha$ would be given either by \[ (1) \, \myinline{
}.\] We may assume that its action on the pairs $\mathcal{R}_{1}$ and $\mathcal{R}_{2}$ is as in (1) since we can compose an automorphism whose action is as in (2) with the element of $A_{0}$ that corresponds to $\gamma_{1}+\gamma_{2}=(0,1,1,0,0)+(1,0,1,0,0)=(1,1,0,0,0)$ to obtain an automorphism whose action is that in (1). On the pairs $\mathcal{R}_{4}$ and $\mathcal{R}_{5}$, the action of $\alpha$ is given either by \[ (3) \, \myinline{%
}.\] As before, we may assume that the action is as in (3) by composing an automorphism whose action is as in (4) with the element of $A_{0}$ corresponding to $\gamma_{3}=(0,0,0,1,1)$ to obtain an automorphism whose action is that in (3). Summarising, we have to study only two cases: \begin{center} %
 \end{center}
In both cases (a) and (b), $f$ and $\overline{f}$ are preserved by looking at the pairs $\mathcal{R}_{4}, \mathcal{R}_{5}$, and hence $f+\overline{f}$ is preserved and so are $f+2\overline{f}, 2f+\overline{f}$.\\ Let us show that case (a) does not exist. Looking at the pair $\mathcal{R}_{3}$ we see that $f+\overline{f}-E_{p}-E_{s}, f+\overline{f}-E_{q}-E_{r}$ are preserved, and then $E_{p}+E_{s}$ and $E_{q}+E_{r}$ are preserved, while the action on the pairs $\mathcal{R}_{1}$ and $\mathcal{R}_{2}$ gives that $\alpha$ interchanges $E_{p}+E_{q}$ with $E_{p}+E_{r}$ and $E_{r}+E_{s}$ with $E_{q}+E_{s}$. This implies that $E_{p}, E_{s}$ are fixed and $E_{q}, E_{r}$ are exchanged. So $\alpha$ comes from a real automorphism $\alpha'$ of $Q_{3,1}$ which fixes $p$ and $s$, interchanges $q$ with $r$, and preserves the fibrations $f, \overline{f}$ over $\p^{1}_{\mathbb{C}} \times \p^{1}_{\mathbb{C}}$. Let us see that such an $\alpha'$ does not exist. The automorphism $\alpha'$ is given by $(x,y) \mapsto (Ax,\overline{A}y)$ where $A \in \PGL_{2}(\mathbb{C})$ with $\alpha'(p)=p$, $\alpha':q \leftrightarrow r$ and $\alpha'(s)=s$. Then, under the choice of the points $p=([1:0],[1:0]), q=([0:1],[0:1]), r=([1:1],[1:1])$ and $s=([\lambda:1],[\overline{\lambda}:1])$ for some $\lambda \in \mathbb{C} \setminus \mathbb{R}$ (Lemma \ref{lem:Lemma_choice_of_points_Q31(4,0)}), we have $A=
$ and $\lambda=\dfrac{1}{2}$, which gives a contradiction on the parameter $\lambda$. So an automorphism of $X$ of type (a) does not exist.\\ In case (b), $\alpha$ is not even an automorphism of the Picard group because the matrix corresponding to an action described in (b) with respect to the basis $\{f,\overline{f},E_{p},E_{q},E_{r},E_{s}\}$ is \small\[ %
 \notin \GL_{6}(\mathbb{Z}).\]

Now, we prove that automorphisms of type $(45)$ are not in the image of $\rho$ and we proceed in the same way as we did for automorphisms of type $(12)$. The action of an automorphism of $X$ of type $(45)$ on the pairs $\mathcal{R}_{4}$ and $\mathcal{R}_{5}$ is given either by \[ (1) \, \myinline{%
}.\] Composing with the element of $A_{0}$ that corresponds to $\gamma_{3}=(0,0,0,1,1)$, we may assume that the action is as in (1). With respect to the action on the first three pairs $\mathcal{R}_{1}$, $\mathcal{R}_{2}$ and $\mathcal{R}_{3}$, we assume that the action on $\mathcal{R}_{1}$ and $\mathcal{R}_{3}$ is the identity since we can compose with automorphisms corresponding to $\gamma_{1}=(0,1,1,0,0)$ or to $\gamma_{1}+\gamma_{2}=(1,1,0,0,0)$. Then, we have two cases to focus on: \begin{center}
%
\end{center} 
Case (a) corresponds to an automorphism which interchanges $f$ with $\overline{f}$ and fixes $E_{p}$, $E_{q}$, $E_{r}$ and $E_{s}$. It would be the lift of an automorphism of $Q_{3,1}$ fixing the four points $p$, $q$, $r$ and $s$, which does not exist. Case (b) is not even an automorphism of the Picard group because the matrix corresponding to its action is \small\[ %
 \notin \GL_{6}(\mathbb{Z}).\]

Then, we show that there is an automorphism which acts as $(123)$ if and only if $\lambda=e^{\pm i \frac{\pi}{3}}$. If there is an automorphism $\alpha$ of $X$ acting as the cycle $(123)$ on the five pairs of conic bundles, then by composing with the elements of $A_{0}$ corresponding to $\gamma_{1}=(0,1,1,0,0)$, $\gamma_{2}=(1,0,1,0,0)$ and $\gamma_{1}+\gamma_{2}=(1,1,0,0,0)$, we may assume that its action on the pairs $\mathcal{R}_{1}, \mathcal{R}_{2}, \mathcal{R}_{3}$ is given by \[ (1) \, \myinline{%
}. \]
On the pairs $\mathcal{R}_{4}$ and $\mathcal{R}_{5}$ the action of $\alpha$ is given either by \[ (3) \, \myinline{%
}. \] Again, we may assume that its action is as in (3) since we can compose an automorphism whose action is as in (4) with the element of $A_{0}$ corresponding to $\gamma_{3}=(0,0,0,1,1)$ to obtain an automorphism whose action is that in (3). Summarising, we have to focus on the following two cases: \begin{center}
%
 \end{center}
In both cases (a) and (b), $f$ and $\overline{f}$ are preserved by looking at pairs $\mathcal{R}_{4}$ and $\mathcal{R}_{5}$, and hence $f+\overline{f}$ is preserved. In case (a), looking at pairs $\mathcal{R}_{1}, \mathcal{R}_{2}, \mathcal{R}_{3}$, we see that $E_{p}+E_{q} \mapsto E_{p}+E_{r} \mapsto E_{p}+E_{s} \mapsto E_{p}+E_{q}$ and $E_{r}+E_{s} \mapsto E_{q}+E_{s} \mapsto E_{q}+E_{r} \mapsto E_{r}+E_{s}$. This implies that $E_{p}$ is fixed and that $E_{q} \mapsto E_{r} \mapsto E_{s} \mapsto E_{q}$. So an automorphism $\alpha$ of type $(123)$ in case (a) comes from an automorphism  $\alpha'$ of $\p^{1}_{\mathbb{C}} \times \p^{1}_{\mathbb{C}}$ defined over $\mathbb{R}$ preserving the two rulings, fixing $p=([1:0],[1:0])$ and permuting $q=([0:1],[0:1])$, $r=([1:1],[1:1])$, $s=([\lambda:1],[\overline{\lambda}:1])$ cyclically. The automorphism $\alpha'$ is then given by $(x,y) \mapsto (Ax,\overline{A}y)$ where $A \in \PGL_{2}(\mathbb{C})$ is satisfying $$A 
.$$ This implies that $$A=%
$$ with $\lambda$ satisfying the relation $\lambda^{2}-\lambda+1=0$, and this proves that an automorphism of type (a) exists if and only if $\lambda=e^{\pm i \frac{\pi}{3}}$. An automorphism of $X$ in case (b) does not exist since the matrix of its action on the Picard group with respect to the basis $\lbrace f,\overline{f},E_{p},E_{q},E_{r},E_{s} \rbrace$ is \small\[ %
 \notin \GL_{6}(\mathbb{Z}).\]

Finally, we check that there is an automorphism of $X$ which acts as $(12)(45)$ if and only if $\lambda+\overline{\lambda}=1$. As before, we can see that the actions of automorphisms of type $(12)(45)$ are, up to composition with an element of $A_{0}$ corresponding to $\gamma_{1}+\gamma_{2}=(1,1,0,0,0)$ or to $\gamma_{3}=(0,0,0,1,1)$, of the form
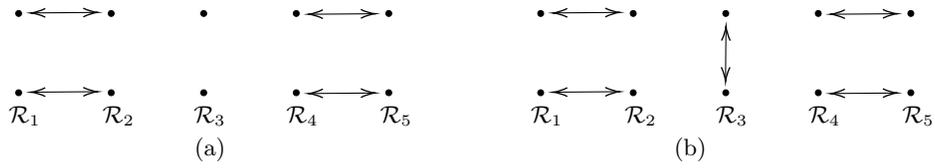
\begin{center} %
 \end{center}
In both cases (a) and (b), looking at the pairs $\mathcal{R}_{4}$ and $\mathcal{R}_{5}$ we see that $f$ and $\overline{f}$ are exchanged, and so $f+\overline{f}$ is preserved. Looking at pairs $\mathcal{R}_{1}, \mathcal{R}_{2}$, we see that $f+\overline{f}-E_{p}-E_{q}$ is exchanged with $f+\overline{f}-E_{p}-E_{r}$, and so are $f+\overline{f}-E_{r}-E_{s}$ and $f+\overline{f}-E_{q}-E_{s}$. This implies that $E_{p}+E_{q}$ and $E_{p}+E_{r}$ are exchanged, and so are $E_{r}+E_{s}$ and $E_{q}+E_{s}$. In case (a), looking now at the pair $\mathcal{R}_{3}$ we see that $f+\overline{f}-E_{p}-E_{s}$ and $f+\overline{f}-E_{q}-E_{r}$ are preserved, so $E_{p}+E_{s}$ and $E_{q}+E_{r}$ are preserved. This second point implies that $E_{p}, E_{s}$ are preserved and that $E_{q}$ and $E_{r}$ are exchanged. So an automorphism of type $(12)(45)$ in case (a) comes from a (real) automorphism $\alpha'$ of $Q_{3,1}$ exchanging the two fibrations of $\p^{1}_{\mathbb{C}} \times \p^{1}_{\mathbb{C}}$, fixing the points $p$ and $s$ and exchanging $q$ with $r$. So $\alpha'$ is given by $\alpha' : (x,y) \mapsto (\overline{A}y,Ax)$, where $A \in \PGL_{2}(\mathbb{C})$ is satisfying $$A \begin{bmatrix} 1 \\ 0 \end{bmatrix} = \begin{bmatrix} 1 \\ 0 \end{bmatrix}, \, A \begin{bmatrix} 0 \\ 1 \end{bmatrix} = \begin{bmatrix} 1 \\ 1 \end{bmatrix}, \, A \begin{bmatrix} 1 \\ 1 \end{bmatrix} = \begin{bmatrix} 0 \\ 1 \end{bmatrix}.$$ This implies that $$A = \begin{bmatrix} -1 & 1 \\ 0 & 1 \end{bmatrix}.$$ Since $\alpha'$ fixes $s$, then $$\begin{bmatrix} -1 & 1 \\ 0 & 1 \end{bmatrix} \begin{bmatrix} \lambda \\ 1 \end{bmatrix} = \begin{bmatrix} \overline{\lambda} \\ 1 \end{bmatrix} = \begin{bmatrix} 1-\lambda \\ 1 \end{bmatrix},$$ and hence $\lambda + \overline{\lambda}=1$. Therefore an automorphism of type (a) exists if and only if $\lambda + \overline{\lambda} = 1$, and it is given by the lift of the involution $ \alpha' : ([u_{0}:u_{1}],[v_{0}:v_{1}]) \mapsto ([v_{0}-v_{1}:-v_{1}],[u_{0}-u_{1}:-u_{1}])$ of $Q_{3,1}$. Case (b) is not possible because the matrix of the action described in (b) on the Picard group of $X$ with respect to the basis $\lbrace f,\overline{f},E_{p},E_{q},E_{r},E_{s} \rbrace$ is \small\[ \begin{pmatrix} 0 & 1 & 0 & 0 & 0 & 0 \\ 1 & 0 & 0 & 0 & 0 & 0 \\ 0 & 0 & \frac{1}{2} & \frac{1}{2} & \frac{1}{2} & -\frac{1}{2} \\ 0 & 0 & \frac{1}{2} & -\frac{1}{2} & \frac{1}{2} & \frac{1}{2} \\ 0 & 0 & \frac{1}{2} & \frac{1}{2} & -\frac{1}{2} & \frac{1}{2} \\ 0 & 0 & -\frac{1}{2} & \frac{1}{2} & \frac{1}{2} & \frac{1}{2} \end{pmatrix} \notin \GL_{6}(\mathbb{Z}),\] which shows that it is not even an automorphism of $\Pic(X)$. 
\end{proof}

\begin{remark} In this case, and from the description of families of automorphisms in $\Aut_{\mathbb{R}}(X)$ given in Proposition \ref{Prop:Proposition_Aut(X)_Q31(4,0)}, one can see that families of quartic del Pezzo surfaces $X$ of type $Q_{3,1}(4,0)$ are such that the action of $\Gal(\mathbb{C}/\mathbb{R})$ on $\Pic(X)$ (see Figure \ref{Fig:Figure_3_Galois_action_on_conic_bundles_Q31(4,0)}) is not realized by an automorphism of the surface.
\end{remark}

\subsubsection{The del Pezzo surfaces obtained by blowing up $Q_{2,2}$}

\subsubsection{Case 4: $X \cong Q_{2,2}(0,2)$} Let $X \cong Q_{2,2}(0,2)$ be the blow-up of $Q_{2,2} \cong \p^{1}_{\mathbb{R}} \times \p^{1}_{\mathbb{R}}$ at two pairs of complex conjugate points $\lbrace p,\overline{p} \rbrace$, $\lbrace q,\overline{q} \rbrace$.\\ There are sixteen $(-1)$-curves on $X_{\mathbb{C}}$: the exceptional divisors above the blown up points denoted by $E_{p}, E_{\overline{p}}, E_{q}, E_{\overline{q}}$; the strict transforms of the fibres $f_{1}, f_{2}$ passing through one of the four points that we denote by $f_{1p}, f_{1q}, f_{2p}, f_{2q}, f_{1\overline{p}}, f_{1\overline{q}}, f_{2\overline{p}}, f_{2\overline{q}}$; and the strict transforms of the curves of bidegree (1,1) of $\p^{1}_{\mathbb{C}} \times \p^{1}_{\mathbb{C}}$ (i.e. linearly equivalent to $f_{1}+f_{2}$) passing through three of the four points and denoted by $f_{p\overline{p}q}, f_{p\overline{p}\overline{q}}, f_{pq\overline{q}}, f_{\overline{p}q\overline{q}}$. These $(-1)$-curves form the singular fibres of the ten conic bundle structures on $X_{\mathbb{C}}$ with four singular fibres each. These conic bundles are given by the linear systems associated to the following classes of divisors:\\
\begin{tabular}{p{6cm}p{7cm}}
{\begin{align}
f_{1}+f_{2}-E_{p}-E_{\overline{p}}  \\  f_{1}+f_{2}-E_{p}-E_{q}  \\  f_{1}+f_{2}-E_{p}-E_{\overline{q}}  \\  f_{1}+f_{2}-E_{q}-E_{\overline{q}}  \\  f_{1}+f_{2}-E_{\overline{p}}-E_{\overline{q}}   
\end{align}}
&%
{\begin{align}
f_{1}+f_{2}-E_{\overline{p}}-E_{q}  \\  f_{1} \\  f_{2}  \\  f_{1}+2f_{2}-E_{p}-E_{\overline{p}}-E_{q}-E_{\overline{q}}  \\  2f_{1}+f_{2}-E_{p}-E_{\overline{p}}-E_{q}-E_{\overline{q}}
\end{align}}
\end{tabular}\\The anticanonical divisor class of $X$ is $-K_{X}=2f_{1}+2f_{2}-E_{p}-E_{\overline{p}}-E_{q}-E_{\overline{q}}$. We collect again these conic bundles in pairs $\mathcal{R}_{i}=\lbrace \mathcal{C}_{i},\mathcal{C}'_{i} \rbrace$, such that $\mathcal{C}_{i}+\mathcal{C}'_{i}=-K_{X}$ for $i=1,\dots,5$. One has
\begin{align*}
\mathcal{R}_{1} &= \lbrace f_{1}+f_{2}-E_{p}-E_{\overline{p}} , f_{1}+f_{2}-E_{q}-E_{\overline{q}} \rbrace,\\
\mathcal{R}_{2} &= \lbrace f_{1} , f_{1}+2f_{2}-E_{p}-E_{\overline{p}}-E_{q}-E_{\overline{q}} \rbrace,\\
\mathcal{R}_{3} &= \lbrace f_{2} , 2f_{1}+f_{2}-E_{p}-E_{\overline{p}}-E_{q}-E_{\overline{q}} \rbrace,\\
\mathcal{R}_{4} &= \lbrace f_{1}+f_{2}-E_{p}-E_{q} , f_{1}+f_{2}-E_{\overline{p}}-E_{\overline{q}} \rbrace,\\
\mathcal{R}_{5} &= \lbrace f_{1}+f_{2}-E_{p}-E_{\overline{q}} , f_{1}+f_{2}-E_{\overline{p}}-E_{q} \rbrace.
\end{align*}
The action of the standard anti-holomorphic involution on the above five pairs is then represented with arrows in Figure \ref{Fig:Figure_4_Galois_action_on_conic_bundles_Q22(0,2)} below.

\begin{figure}[h]
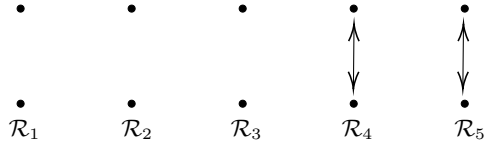

\centering

\caption{The action of $\Gal(\mathbb{C}/\mathbb{R})$ on the five pairs of conic bundles.}
\label{Fig:Figure_4_Galois_action_on_conic_bundles_Q22(0,2)}
\end{figure} As a consequence of the given Galois action (Figure \ref{Fig:Figure_4_Galois_action_on_conic_bundles_Q22(0,2)}) and as stated in \cite[Proposition 6.8]{yas22}, we see that $A_{0}$ lies inside the group $\langle \gamma_{1}=(1,0,1,0,0),\gamma_{2}=(0,0,1,1,0),\gamma_{3}=(0,0,1,0,1),\gamma=(0,1,1,0,0) \rangle \cong (\Z/2\Z)^{4} \subset (\Z/2\Z)^{5}$ and that $A'$ is a subgroup of $\langle (123),(12) \rangle \times \langle (45) \rangle \cong \Sym_{3} \times \Z/2\Z \cong \D_{6}$ isomorphic to $\lbrace \id \rbrace$, $\Z/2\Z$, $\Z/3\Z$ or $\Sym_{3}$ (see list \eqref{list_of_possible_group_for_Im_Aut(X_C)}).
Let us go a little further by describing geometrically the families of automorphisms in $\Aut_{\mathbb{R}}(X)$.

\begin{lemma}
Let $p, q \in \p^{1}_{\mathbb{C}} \times \p^{1}_{\mathbb{C}} \cong (Q_{2,2})_{\mathbb{C}}$ be two distinct imaginary non conjugate points such that the blow-up of $p, \overline{p}, q, \overline{q}$ is a real del Pezzo surface. Then up to automorphisms of $\p^{1}_{\mathbb{R}} \times \p^{1}_{\mathbb{R}}$, the points $p$ and $q$ can be chosen to be $p=([1:i],[1:i])$ and $q=([1:k_{1}i],[1:k_{2}i])$ for some distinct $k_{1}, k_{2} \in \, ]0,1[$, respectively.
\label{lem:Lemma_choice_of_points_Q22(0,2)} 
\end{lemma}

\begin{proof}
We may assume that $p=([1:i],[1:i])$ and $\overline{p}=([1:-i],[1:-i])$. We then consider $(A,B)=\left( \begin{pmatrix} a & b \\ -b & a \end{pmatrix} , \begin{pmatrix} c & d \\ -d & c \end{pmatrix} \right) \in \PGL_{2}(\mathbb{R}) \times \PGL_{2}(\mathbb{R})$ where $a, b, c, d \in \mathbb{R}^{*}$. It defines an automorphism of $Q_{2,2} \cong \p^{1}_{\mathbb{R}} \times \p^{1}_{\mathbb{R}}$ that fixes the points $p$ and $\overline{p}$. Now we can assume that $q=([1:\alpha+i\beta],[1:\gamma+i\delta])$ with $\alpha, \beta, \gamma, \delta \in \mathbb{R}^{*}$. We need to find $a, b, c, d \in \mathbb{R}^{*}$ such that $(A,B)$ sends the point $q$ onto the point $([1:k_{1}i],[1:k_{2}i])$, where $k_{1}, k_{2}$ are two distinct real parameters to identify, that is such that $$A \begin{bmatrix} 1 \\ \alpha + i \beta \end{bmatrix} = \begin{bmatrix}
1 \\ k_{1}i \end{bmatrix} \,\,\, \text{and} \,\,\, B \begin{bmatrix} 1 \\ \gamma + i \delta \end{bmatrix} = \begin{bmatrix} 1 \\ k_{2}i \end{bmatrix}.$$ The first condition (on $a, b, k_{1}$) implies (1): $\beta a = k_{1}(a+\alpha b)$ and (2) : $-b+\alpha a = -\beta b k_{1}$, which gives (1) : $k_{1}=\frac{\beta a}{a+\alpha b}$ and then the quadratic equation in $a$, (3) : $\alpha a^{2}+(-b+\alpha^{2}b+\beta^{2}b^{2})a-\alpha \beta^{2} = 0$. For a choice of $b\ggg 0$ the discriminant of (3) is positive, so that we have a real root $a(b,\alpha,\beta)$ and such that $0 < \widehat{k_{1}}:=\frac{\beta a(b,\alpha,\beta)}{a(b,\alpha,\beta)+\alpha b} < 1$. Consequently, the relations (1) and (2) : $-b+\alpha a(b,\alpha,\beta) = \frac{-b \beta^{2} a(b,\alpha,\beta)}{a(b,\alpha,\beta)+\alpha b} = -b \beta \widehat{k_{1}}$ are satisfied. The same reasoning applies for $c, d, k_{2}$, so that we have found an automorphism of $\p^{1}_{\mathbb{R}} \times \p^{1}_{\mathbb{R}}$ fixing $p, \overline{p}$, and sending $q, \overline{q}$ onto $([1:\widehat{k_{1}}i],[1:\widehat{k_{2}}i]), ([1:-\widehat{k_{1}}i],[1:-\widehat{k_{2}}i])$, respectively, where $\widehat{k_{1}}, \widehat{k_{2}} \in \, ]0,1[$.    
\end{proof}

\begin{proposition}
Let $X \cong Q_{2,2}(0,2)$ be a real del Pezzo surface of degree $4$.
\begin{enumerate}
\item The kernel $A_{0}$ of the homomorphism $\rho : \Aut_{\mathbb{R}}(X) \rightarrow \Sym_{5}$ is isomorphic to $(\Z/2\Z)^{4}$, and it is generated by the elements $\gamma_{1}=(1,0,1,0,0), \gamma_{2}=(0,0,1,1,0), \gamma_{3}=(0,0,1,0,1)$ and $\gamma=(0,1,1,0,0)$.
\label{it:item_(1)_Proposition_Aut(X)_Q22(0,2)}
\item The image $A'$ of the homomorphism $\rho$ is either $\langle (12)(45) \rangle \cong \Z/2\Z$ or $\lbrace \id \rbrace$. Moreover, the former happens if and only if $k_{1}^{2}k_{2}-2k_{1}+k_{2}=0$.
\label{it:item_(2)_Proposition_Aut(X)_Q22(0,2)}  
\end{enumerate}
Furthermore, the elements $\gamma_{1}, \gamma_{2}, \gamma_{3}$ and $\gamma$ are realized in $A_{0}$ as the lifts of some birational involutions of $Q_{2,2}$ with base-points $p, \overline{p}, q, \overline{q}$. The element $\tau=(12)(45)$ can be realized as the lift on $X$ of a birational involution of $Q_{2,2}$ with base-points $p, \overline{p}$. 
\label{Prop:Proposition_Aut(X)_Q22(0,2)}
\end{proposition}

\begin{proof}
(\ref{it:item_(1)_Proposition_Aut(X)_Q22(0,2)}) We have to show the existence of automorphisms of $X$ corresponding to the elements $\gamma_1, \gamma_2, \gamma_3$ and $\gamma$ in $(\Z/2\Z)^{5}$.\\
Consider the real birational involution $\phi_{1} : ([u_{0}:u_{1}],[v_{0}:v_{1}]) \dashmapsto ( [-u_{0}:u_{1}],[k_{1}(1-k_{1}k_{2})u_{0}^{2}v_{1}+k_{2}(k_{1}^{2}-1)u_{0}u_{1}v_{0}+(k_{1}-k_{2})u_{1}^{2}v_{1}:k_{2}((k_{1}^{2}-1)u_{0}u_{1}v_{1}+k_{1}(k_{1}-k_{2})u_{0}^{2}v_{0}+(1-k_{1}k_{2})u_{1}^{2}v_{0})] )$ of $\p^{1}_{\mathbb{R}} \times \p^{1}_{\mathbb{R}}$, whose base-points are $p, \overline{p}, q, \overline{q}$. It lifts by blowing up these points to a biregular morphism of $X$ defined over $\mathbb{R}$, and whose action on the Picard group of $X$ with respect to the basis $\lbrace f_{1},f_{2},E_{p},E_{\overline{p}},E_{q},E_{\overline{q}} \rbrace$ is given by the matrix \[ \small \begin{pmatrix} 1 & 2 & 1 & 1 & 1 & 1 \\ 0 & 1 & 0 & 0 & 0 & 0 \\ 0 & -1 & 0 & -1 & 0 & 0 \\ 0 & -1 & -1 & 0 & 0 & 0 \\ 0 & -1 & 0 & 0 & 0 & -1 \\ 0 & -1 & 0 & 0 & -1 & 0 \end{pmatrix}.\] It follows that this automorphism is in the kernel of $\rho$ and its action on the pairs of conic bundle structures is the one of the element $\gamma_{1}=(1,0,1,0,0)$. We then consider the real birational involution $\phi_{2} : ([u_{0}:u_{1}],[v_{0}:v_{1}]) \dashmapsto ([A_{0}(u_{0},u_{1}):A_{1}(u_{0},u_{1})],[B_{0}(u_{0},u_{1},v_{0},v_{1}):B_{1}(u_{0},u_{1},v_{0},v_{1})])$ of $\p^{1}_{\mathbb{R}} \times \p^{1}_{\mathbb{R}}$, with
\begin{align*}
A_{0}(u_{0},u_{1}) &= u_{1},\\
A_{1}(u_{0},u_{1}) &= -k_{1}u_{0},\\
B_{0}(u_{0},u_{1},v_{0},v_{1}) &= (k_{1}^{2}-1)u_{0}u_{1}v_{1}+k_{1}(k_{1}-k_{2})u_{0}^{2}v_{0}+(1-k_{1}k_{2})u_{1}^{2}v_{0},\\
B_{1}(u_{0},u_{1},v_{0},v_{1}) &= k_{1}(1-k_{1}k_{2})u_{0}^{2}v_{1}+k_{2}(k_{1}^{2}-1)u_{0}u_{1}v_{0}+(k_{1}-k_{2})u_{1}^{2}v_{1},
\end{align*} and whose base-points are again $p, \overline{p}, q$ and $\overline{q}$. By blowing up these points it lifts to an automorphism of $X$ defined over $\mathbb{R}$ whose action on $\Pic(X)$ with respect to the basis $\lbrace f_{1},f_{2},E_{p},E_{\overline{p}},E_{q},E_{\overline{q}} \rbrace$ is given by the matrix \[ \small \begin{pmatrix} 1 & 2 & 1 & 1 & 1 & 1 \\ 0 & 1 & 0 & 0 & 0 & 0 \\ 0 & -1 & 0 & 0 & -1 & 0 \\ 0 & -1 & 0 & 0 & 0 & -1 \\ 0 & -1 & -1 & 0 & 0 & 0 \\ 0 & -1 & 0 & -1 & 0 & 0 \end{pmatrix}.\] It follows that this automorphism is in the kernel of $\rho$ and its action on the pairs of conic bundle structures is given by $\langle \gamma_{2}=(0,0,1,1,0) \rangle$. One can show in a similar way the existence of an automorphism of $X$ corresponding to the element $ \gamma_{3}=(0,0,1,0,1)$ by lifting a well-chosen real birational involution of $\p^{1}_{\mathbb{R}} \times \p^{1}_{\mathbb{R}}$, since the action of such an automorphism on the Picard group of $X$ is given by the matrix \[ \small \begin{pmatrix} 1 & 2 & 1 & 1 & 1 & 1 \\ 0 & 1 & 0 & 0 & 0 & 0 \\ 0 & -1 & 0 & 0 & 0 & -1 \\ 0 & -1 & 0 & 0 & -1 & 0 \\ 0 & -1 & 0 & -1 & 0 & 0 \\ 0 & -1 & -1 & 0 & 0 & 0 \end{pmatrix}\] with respect to the basis $\lbrace f_{1},f_{2},E_{p},E_{\overline{p}},E_{q},E_{\overline{q}} \rbrace$. Finally, consider the real birational involution $\phi : ([u_{0}:u_{1}],[v_{0}:v_{1}]) \dashmapsto ([C_{0}(u_{0},u_{1},v_{0},v_{1}):C_{1}(u_{0},u_{1},v_{0},v_{1})],[D_{0}(u_{0},u_{1},v_{0},v_{1}):D_{1}(u_{0},u_{1},v_{0},v_{1})])$ of $\p^{1}_{\mathbb{R}} \times \p^{1}_{\mathbb{R}}$, where
\begin{align*}
C_{0}(u_{0},u_{1},v_{0},v_{1}) &= k_{1}(1-k_{2}^{2})u_{0}v_{0}v_{1} + k_{2}(k_{1}k_{2}-1)u_{1}v_{0}^{2} + (k_{1}-k_{2})u_{1}v_{1}^{2},\\
C_{1}(u_{0},u_{1},v_{0},v_{1}) &= -k_{1}((k_{1}k_{2}-1)u_{0}v_{1}^{2} + k_{2}(k_{1}-k_{2})u_{0}v_{0}^{2} + (1-k_{2}^{2})u_{1}v_{0}v_{1}),\\
D_{0}(u_{0},u_{1},v_{0},v_{1}) &= k_{1}(k_{1}k_{2}-1)u_{0}^{2}v_{1} + k_{2}(1-k_{1}^{2})u_{0}u_{1}v_{0} + (k_{2}-k_{1})u_{1}^{2}v_{1},\\
D_{1}(u_{0},u_{1},v_{0},v_{1}) &= k_{2}((k_{1}^{2}-1)u_{0}u_{1}v_{1} + k_{1}(k_{1}-k_{2})u_{0}^{2}v_{0} + (1-k_{1}k_{2})u_{1}^{2}v_{0}),
\end{align*} and whose base-points are $p, \overline{p}, q$ and $\overline{q}$. It lifts by blowing up these four points to an automorphism of $X$ defined over $\mathbb{R}$ which is in the kernel of $\rho$ and whose action on the pairs $\mathcal{R}_{i}$'s is the one of the element $\gamma=(0,1,1,0,0)$. Indeed, its action on the Picard group of $X$ with respect to the basis $\lbrace f_{1},f_{2},E_{p},E_{\overline{p}},E_{q},E_{\overline{q}} \rbrace$ is given by the matrix \[ \small \begin{pmatrix} 1 & 2 & 1 & 1 & 1 & 1 \\ 2 & 1 & 1 & 1 & 1 & 1 \\ -1 & -1 & 0 & -1 & -1 & -1 \\ -1 & -1 & -1 & 0 & -1 & -1 \\ -1 & -1 & -1 & -1 & 0 & -1 \\ -1 & -1 & -1 & -1 & -1 & 0 \end{pmatrix}.\] This yields the claim of point (\ref{it:item_(1)_Proposition_Aut(X)_Q22(0,2)}).

(\ref{it:item_(2)_Proposition_Aut(X)_Q22(0,2)}) Recall that $A'$ is a subgroup of $\langle (123),(12) \rangle \times \langle (45) \rangle \cong \D_{6}$ isomorphic to $\lbrace \id \rbrace, \Z/2\Z, \Z/3\Z$ or $\Sym_{3}$. We show that the elements $(12), (45)$ and $(123)$ do not belong to the image, while the element $(12)(45)$ does if and only if $k_{1}^{2}k_{2}-2k_{1}+k_{2}=0$.\\ If there were an automorphism of $X$ exchanging the pairs $\mathcal{R}_{1}$ and $\mathcal{R}_{2}$, then its action would be given either by \[ (1) \, \myinline{

},\] and we may assume that the action is as in (1) since we can compose an automorphism whose action is as in (2) with an element of $A_{0}$ that corresponds to $\gamma_{1}+\gamma=(1,1,0,0,0)$ to obtain an automorphism whose action is that in (1). To study the action of such an automorphism on the pairs $\mathcal{R}_{3}, \mathcal{R}_{4}, \mathcal{R}_{5}$, we first look at the pairs $(\mathcal{R}_{4},\mathcal{R}_{5})$ and we may assume that its action on it is given by \, (3) \, \myinline{%
} since we can compose an automorphism whose action is given by \, (4) \, \myinline{%
} with an element of $A_{0}$ corresponding to $\gamma_{2}+\gamma_{3}=(0,0,0,1,1)$ to obtain an automorphism whose action is that in (3). We can argue the same way about its action on the pairs $(\mathcal{R}_{3},\mathcal{R}_{5})$ and $(\mathcal{R}_{3},\mathcal{R}_{4})$ by composing with automorphisms of $A_{0}$ corresponding to $\gamma_{3}=(0,0,1,0,1)$ and $\gamma_{2}=(0,0,1,1,0)$, respectively. 
Summarising, we can focus on the following two cases: \begin{center} %
 \end{center}
In case (a), the action of such an automorphism on the Picard group of $X$ with respect to the basis $\lbrace f_{1},f_{2},E_{p},E_{\overline{p}},E_{q},E_{\overline{q}} \rbrace$ is given by the matrix \[\small %
;\]
that is, if it exists, an automorphism of $X$ of type (a) comes from the lift of a birational involution of $\p^{1}_{\mathbb{R}} \times \p^{1}_{\mathbb{R}}$ that fixes $q, \overline{q}$, and with (two) base-points $p, \overline{p}$. One can check by explicit computations that such a birational involution does not exist; otherwise it would imply $k_{2}=\pm i$ and $k_{1}=\pm 1$, which is not possible under the choice of the points $p, \overline{p}, q, \overline{q}$ (see Lemma \ref{lem:Lemma_choice_of_points_Q22(0,2)}). Case (b) is not even an automorphism of the Picard group because the matrix corresponding to an action described in (b), with respect to the basis $\lbrace f_{1},f_{2},E_{p},E_{\overline{p}},E_{q},E_{\overline{q}} \rbrace$, is \small\[ 
 \notin \GL_{6}(\mathbb{Z}).\] Therefore, an automorphism that acts as $(12)$ does not belong to the image of $\rho$.

Now, we proceed in the same way for automorphisms acting as $(45)$ on the pairs of conic bundle structures: since we can compose with automorphisms in $A_{0}$ corresponding to $\gamma_{2}+\gamma_{3}=(0,0,0,1,1), \gamma_{1}+\gamma=(1,1,0,0,0), \gamma_{1}=(1,0,1,0,0)$ and $\gamma=(0,1,1,0,0)$, we may assume that the actions of such automorphisms of $X$ are of the form \begin{center} %
 \end{center}
In both cases (a) and (b), looking at pairs $\mathcal{R}_{2}, \mathcal{R}_{3}$, we see that $f_{1}, f_{2}$ are preserved and that $E_{p}+E_{\overline{p}}+E_{q}+E_{\overline{q}}$ is therefore preserved. In case (a), we see that $E_{p}+E_{\overline{p}}$ and $E_{q}+E_{\overline{q}}$ are preserved by looking at pair $\mathcal{R}_{1}$, and that $E_{p}+E_{q}$ and $E_{p}+E_{\overline{q}}$ are exchanged as are $E_{\overline{p}}+E_{\overline{q}}$ and $E_{\overline{p}}+E_{q}$ by looking at pairs $\mathcal{R}_{4}, \mathcal{R}_{5}$. This gives that an automorphism $\alpha$ of $X$ of type (a) comes from an automorphism $\alpha'$ of $\p^{1}_{\mathbb{R}} \times \p^{1}_{\mathbb{R}}$ preserving the fibrations, fixing the points $p, \overline{p}$, and exchanging $q$ with $\overline{q}$. So $\alpha'$ would be of the form $\alpha' : (x,y) \mapsto (Ax,By)$ with $(A,B) \in (\PGL_{2}(\mathbb{R}))^{2}$ and such that $A, B$ satisfy $$A 
 \right),$$ and since $\alpha'$ interchanges $q$ with $\overline{q}$, then we get $$A %
.$$ Hence, $a_{1}=b_{1}=0$ and $k_{1}^{2}=k_{2}^{2}=-1$, which is not possible by assumption on the real parameters $k_{1}, k_{2}$. Case (b) is not possible because the matrix of the action described in (b) on $\Pic(X)$, with respect to the basis $\lbrace f_{1},f_{2},E_{p},E_{\overline{p}},E_{q},E_{\overline{q}} \rbrace$, is \small\[ %
 \notin \GL_{6}(\mathbb{Z}),\] which shows that it is not even an automorphism of the Picard group.

If there is an automorphism of $X$ acting as $(123)$ on the five pairs of conic bundles, we may assume that its action on the pairs $\mathcal{R}_{1}, \mathcal{R}_{2}, \mathcal{R}_{3}$ is given by \[ (1) \, \myinline{%
},\] by composing with elements of $A_{0}$ corresponding to $\gamma_{1}=(1,0,1,0,0), \gamma_{1}+\gamma=(1,1,0,0,0)$ and $\gamma=(0,1,1,0,0)$. On the pairs $\mathcal{R}_{4}, \mathcal{R}_{5}$, the action of such an automorphism is given either by \[ (3) \,\, \myinline{%
}, \]
and we may assume that the action is as in (3) or as in (5) since we can compose an automorphism whose action is as in (4) or as in (6) with the element of $A_{0}$ corresponding to $\gamma_{2}+\gamma_{3}=(0,0,0,1,1)$ to obtain an automorphism whose action is that in (3) and (5), respectively. To summarize, we need to focus on the following four cases: \begin{center}
%
 \end{center}

In case (a), the induced action on the Picard group of $X$ with respect to the basis $\lbrace f_{1},f_{2},E_{p},E_{\overline{p}},E_{q},E_{\overline{q}} \rbrace$ is given by the matrix \small\[ %
,\] so if it exists, such an automorphism is given by the lift of a real birational transformation of $\p^{1}_{\mathbb{R}} \times \p^{1}_{\mathbb{R}}$ of order $3$ with base-points $p, \overline{p}$, which fixes the points $q, \overline{q}$, and which contracts the curves $f_{1 \overline{p}}, f_{1p}$ onto $p, \overline{p}$, respectively. One can check by explicit computations that such a birational map of $\p^{1}_{\mathbb{R}} \times \p^{1}_{\mathbb{R}}$ exists only if $k_{1}^{2}k_{2}-2k_{1}+k_{2}=0$ and $k_{1} \in \lbrace 0 , \pm i \sqrt{3} , \pm i \frac{\sqrt{3}}{3} \rbrace$, or if $k_{1}^{2}-2k_{1}k_{2}+1=0$ and $k_{1} \in \lbrace \pm i \sqrt{3} , \pm i \frac{\sqrt{3}}{3} \rbrace$, which both give impossible conditions on the real parameters $k_{1}, k_{2}$ (see Lemma \ref{lem:Lemma_choice_of_points_Q22(0,2)}). 
Cases (b) and (c) do not even define automorphisms of the Picard group since the matrices corresponding to the actions described in (b) and (c), with respect to the basis $\lbrace f_{1},f_{2},E_{p},E_{\overline{p}},E_{q},E_{\overline{q}} \rbrace$, are \small\[ \begin{pmatrix} 2 & 1 & 1 & 1 & 1 & 1 \\ 1 & 1 & 0 & 0 & 1 & 1 \\ -1 & 0 & \frac{1}{2} & -\frac{1}{2} & -\frac{1}{2} & -\frac{1}{2} \\ -1 & 0 & -\frac{1}{2} & \frac{1}{2} & -\frac{1}{2} & -\frac{1}{2} \\ -1 & -1 & -\frac{1}{2} & -\frac{1}{2} & -\frac{1}{2} & -\frac{3}{2} \\ -1 & -1 & -\frac{1}{2} & -\frac{1}{2} & -\frac{3}{2} & -\frac{1}{2} \end{pmatrix}, \begin{pmatrix} 0 & 1 & 0 & 0 & 0 & 0 \\ 1 & 1 & 1 & 1 & 0 & 0 \\ 0 & -1 & -\frac{1}{2} & -\frac{1}{2} & -\frac{1}{2} & \frac{1}{2} \\ 0 & -1 & -\frac{1}{2} & -\frac{1}{2} & \frac{1}{2} & -\frac{1}{2} \\ 0 & 0 & -\frac{1}{2} & \frac{1}{2} & \frac{1}{2} & \frac{1}{2} \\ 0 & 0 & \frac{1}{2} & -\frac{1}{2} & \frac{1}{2} & \frac{1}{2} \end{pmatrix} \notin \GL_{6}(\mathbb{Z}).\] In case (d), the matrix of the induced action on $\Pic(X)$ with respect to the basis $\lbrace f_{1},f_{2},E_{p},E_{\overline{p}},E_{q},E_{\overline{q}} \rbrace$ is \small\[ \begin{pmatrix} 2 & 1 & 1 & 1 & 1 & 1 \\ 1 & 1 & 0 & 0 & 1 & 1 \\ -1 & 0 & 0 & 0 & -1 & 0 \\ -1 & 0 & 0 & 0 & 0 & -1 \\ -1 & -1 & -1 & 0 & -1 & -1 \\ -1 & -1 & 0 & -1 & -1 & -1 \end{pmatrix},\] meaning that if such an automorphism of $X$ exists, it comes from a real birational map of $\p^{1}_{\mathbb{R}} \times \p^{1}_{\mathbb{R}}$ of order $6$ with four base-points $p$, $\overline{p}$, $q$, $\overline{q}$, and which contracts the curves $f_{2q}$ onto $p$, $f_{2\overline{q}}$ onto $\overline{p}$, $f_{p q \overline{q}}$ onto $q$ and $f_{\overline{p} q \overline{q}}$ onto $\overline{q}$. We check by explicit computations that such a birational transformation exists only if $k_{1}^{2}k_{2}+k_{2}-2k_{1}=0$ and $k_{1} \in \left\{ 0 , \pm i\sqrt{3} , \pm i \frac{\sqrt{3}}{3} \right\}$, or if $k_{1}^{2}-2k_{1}k_{2}+1=0$ and $k_{1} \in \left\{ \pm i \sqrt{3} , \pm i \frac{\sqrt{3}}{3} \right\}$, which are both impossible conditions on the real parameters $k_{1}$ and $k_{2}$ (see Lemma \ref{lem:Lemma_choice_of_points_Q22(0,2)}). So, an element of type $(123)$ does not belong to the image of $\rho$.

Finally, we can see that the action of an automorphism of type $(12)(45)$ is, up to composition with elements of $A_{0}$ that correspond to $\gamma_{1}+\gamma=(1,1,0,0,0)$ and to $\gamma_{2}+\gamma_{3}=(0,0,0,1,1)$, of the form \begin{center} 
 \end{center}
In case (a), the induced action of such an automorphism on the Picard group of $X$ with respect to the basis $\lbrace f_{1},f_{2},E_{p},E_{\overline{p}},E_{q},E_{\overline{q}} \rbrace$ is given by the matrix \small\[ %
.\] If such an automorphism of $X$ exists, it comes from a real birational involution $\phi$ of $\p^{1}_{\mathbb{R}} \times \p^{1}_{\mathbb{R}}$ with base-points $p, \overline{p}$, and which interchanges $q$ with $\overline{q}$. It contracts the curves $f_{2 p}$ onto $\overline{p}$ and $f_{2 \overline{p}}$ onto $p$. 
We can check by explicit computations that such a birational involution exists exactly if $k_{^1}^{2}k_{2}-2k_{1}+k_{2}=0$ and is of the form $\phi : ([u_{0}:u_{1}],[v_{0}:v_{1}]) \dashmapsto ([u_{1}v_{1}+u_{0}v_{0}:-u_{0}v_{1}+u_{1}v_{0}],[-v_{0}:v_{1}])$. Case (b) is not possible because the matrix of an action described in (b) on $\Pic(X)$ with respect to the basis $\lbrace f_{1},f_{2},E_{p},E_{\overline{p}},E_{q},E_{\overline{q}} \rbrace$ is    
\small\[ \begin{pmatrix} 1 & 2 & 1 & 1 & 1 & 1 \\ 1 & 1 & 1 & 1 & 0 & 0 \\ -1 & -1 & -\frac{1}{2} & -\frac{3}{2} & -\frac{1}{2} & -\frac{1}{2} \\ -1 & -1 & -\frac{3}{2} & -\frac{1}{2} & -\frac{1}{2} & -\frac{1}{2} \\ 0 & -1 & -\frac{1}{2} & -\frac{1}{2} & -\frac{1}{2} & \frac{1}{2} \\ 0 & -1 & -\frac{1}{2} & -\frac{1}{2} & \frac{1}{2} & -\frac{1}{2} \end{pmatrix} \notin \GL_{6}(\mathbb{Z}),\] which shows that it is not even an automorphism of the Picard group. This yields the claim of point (\ref{it:item_(2)_Proposition_Aut(X)_Q22(0,2)}).    
     
\end{proof}

\begin{remark} In this case, one can see from the description of families of automorphisms given in Proposition \ref{Prop:Proposition_Aut(X)_Q22(0,2)} that the action of $\Gal(\mathbb{C}/\mathbb{R})$ on the Picard group of $X$ (see Figure \ref{Fig:Figure_4_Galois_action_on_conic_bundles_Q22(0,2)}), where $X$ is a quartic del Pezzo surface of type $Q_{2,2}(0,2)$, is always realized by an automorphism of $X$ contained in the kernel of $\rho : \Aut_{\mathbb{R}}(X) \rightarrow \Sym_{5}$.\end{remark}

\subsubsection{Case 5: $X \cong Q_{2,2}(4,0)$} To finish, let $X \cong \p^{2}_{\mathbb{R}}(5,0) \cong Q_{2,2}(4,0)$ be the real rational quartic del Pezzo surface obtained by blowing up $\p^{2}_{\mathbb{R}}$ at five real points $p_{1},\dots,p_{5}$ in general position. By first blowing up $\p^{2}_{\mathbb{R}}$ in two of the five points $p_{i}$, and then blowing down the strict transform of the line passing through these two points onto $\p^{1}_{\mathbb{R}} \times \p^{1}_{\mathbb{R}} \cong Q_{2,2}$, we see that $X$ is isomorphic to the blow-up of $Q_{2,2}$ at four general real points denoted by $p, q, r, s$. 
In what follows we will work with one or other of the descriptions of $X$, according to our convenience.

In this case, and using the same notations as in the previous paragraph, a basis of $\Pic(X)$ is given by $\lbrace f_{1},f_{2},E_{p},E_{q},E_{r},E_{s} \rbrace$ and the five pairs of conic bundle structures on $X$ are then given by \begin{align*}
\mathcal{R}_{1} &:= \lbrace f_{1}+f_{2}-E_{p}-E_{q} \ , \ f_{1}+f_{2}-E_{r}-E_{s} \rbrace,\\
\mathcal{R}_{2} &:= \lbrace f_{1}+f_{2}-E_{p}-E_{r} \ , \ f_{1}+f_{2}-E_{q}-E_{s} \rbrace,\\
\mathcal{R}_{3} &:=\lbrace f_{1}+f_{2}-E_{p}-E_{s} \ , \ f_{1}+f_{2}-E_{q}-E_{r} \rbrace,\\
\mathcal{R}_{4} &:= \lbrace f_{1} \ , \ f_{1}+2f_{2}-E_{p}-E_{q}-E_{r}-E_{s} \rbrace,\\
\mathcal{R}_{5} &:= \lbrace f_{2} \ , \ 2f_{1}+f_{2}-E_{p}-E_{q}-E_{r}-E_{s} \rbrace.
\end{align*}

Let us now describe geometrically families of automorphisms in $\Aut_{\mathbb{R}}(X)$.

\begin{lemma}
Let $p, q, r, s \in \p^{1}_{\mathbb{R}} \times \p^{1}_{\mathbb{R}} \cong Q_{2,2}$ be four real points such that the blow-up of $Q_{2,2}$ in $p, q, r, s$ is a real del Pezzo surface. Then, up to an automorphism of $Q_{2,2}$, the points can be chosen to be $p=([1:0],[1:0])$, $q=([0:1],[0:1])$, $r=([1:1],[1:1])$ and $s=([1:\mu_{1}],[1:\mu_{2}])$ for some $\mu_{1}, \mu_{2} \in \mathbb{R} \setminus \lbrace 0,1 \rbrace$, $\mu_{1} \neq \mu_{2}$.
\label{lem:Lemma_choice_of_points_Q22(4,0)} 
\end{lemma}

\begin{proof}
We may assume that $p=([1:0],[1:0])$ and $q=([0:1],[0:1])$, and that the points $r$, $s$ are of the form $r=([1:\lambda],[1:\nu])$, $s=([1:\zeta],[1:\eta])$ for some $\lambda, \nu, \zeta, \eta \in \mathbb{R}^{*}$, since the points are not on the same fibres by any projection by hypothesis. We take $$ (A,B) = \left( \begin{pmatrix} \lambda & 0 \\ 0 & 1 \end{pmatrix},\begin{pmatrix} \nu & 0 \\ 0 & 1\end{pmatrix} \right) \in \, \PGL_{2}(\mathbb{R})^{2}.$$ It defines an automorphism of $\p^{1}_{\mathbb{R}} \times \p^{1}_{\mathbb{R}} \cong Q_{2,2}$ that fixes $p$ and $q$, and that sends $r$ and $s$ onto $([1:1],[1:1])$ and $([1:\mu_{1}],[1:\mu_{2}])$ respectively, where $\mu_{1}=\zeta / \lambda$ and $\mu_{2}=\eta / \nu$.\\ Notice that if $\mu_{1}=0$ or $\mu_{2}=0$, the points $p$ and $s$ are on the same fibres, and so are the points $r$ and $s$ if $\mu_{1}=1$ or $\mu_{2}=1$; and that if $\mu_{1}=\mu_{2}$, the curve of bidegree $(1,1)$ of $\p^{1}_{\mathbb{R},u_{0}u_{1}} \times \p^{1}_{\mathbb{R},v_{0}v_{1}}$ given by $\lbrace u_{0}v_{1}-u_{1}v_{0}=0\rbrace$ is passing through the four points. So, the surface obtained by blowing up the points $p, q, r, s$ is not del Pezzo.
\end{proof}

\begin{proposition}
Let $X \cong \p^{2}_{\mathbb{R}}(5,0) \cong Q_{2,2}(4,0)$ be a real del Pezzo surface of degree $4$.
\begin{enumerate}
\item The kernel $A_{0}$ of the homomorphism $\rho : \Aut_{\mathbb{R}}(X) \rightarrow \Sym_{5}$ is isomorphic to $(\Z/2\Z)^{4}$, and it is generated by the elements $\gamma_{1}=(1,0,0,0,1)$, $\gamma_{2}=(0,1,0,0,1)$, $\gamma_{3}=(0,0,1,0,1)$ and $\gamma=(0,0,0,1,1)$. 
\label{it:item_(1)_Proposition_Aut(X)_Q22(4,0)}
\item The image $A'$ of the homomorphism $\rho$ is either generated by a double-transposition or generated by a double-transposition and a $5$-cycle or it is trivial. More precisely, $A'=\langle (13245) , (12)(45) \rangle \cong \D_{5}$ if and only if $(\mu_{1},\mu_{2}) \in \mathcal{S} = \left\{ \left( \frac{1-\sqrt{5}}{2} , \frac{3-\sqrt{5}}{2} \right) , \left( \frac{1+\sqrt{5}}{2} , \frac{3+\sqrt{5}}{2} \right) \right\}$, $A'=\langle (12)(45) \rangle \cong \Z/2\Z$ if and only if $\mu_{1}+\mu_{2}-\mu_{1}\mu_{2}=0$, $(\mu_{1},\mu_{2}) \notin \mathcal{S}$, and $A'=\lbrace \id \rbrace$ otherwise.
\label{it:item_(2)_Proposition_Aut(X)_Q22(4,0)}
\end{enumerate}
Furthermore, the elements $\gamma_{1}, \gamma_{2}, \gamma_{3}$ and $\gamma$ are realized in $A_{0}$ as the lifts of some birational quadratic involutions of $\p^{2}_{\mathbb{R}}$. The elements $\tau_{1}=(12)(45)$ and $\tau_{2}=(13245)$ can be realized as the lifts on $X$ of an involution of $Q_{2,2}$ and of a birational map of $Q_{2,2}$ of order five with base-points $p$ and $r$, respectively.
\label{Prop:Proposition_Aut(X)_Q22(4,0)}
\end{proposition}

\begin{proof}
(\ref{it:item_(1)_Proposition_Aut(X)_Q22(4,0)}) The proof of the existence of automorphisms of $X$ that correspond to the elements $\gamma_{1}, \gamma_{2}, \gamma_{3}$ and $\gamma$ is analogous to the one of point (\ref{it:item_(2)_Proposition_structure_Aut(X_C)}) of Proposition \ref{prop:Structure_Aut(X_C)} in the complex case. Indeed, using the same notations as in Section \ref{Sec:subsection_classical_statement_over_the_complex}, we can consider the birational quadratic involution $\lambda : [x:y:z] \dashmapsto [ayz:bxz:cxy]$ of $\p^{2}_{\mathbb{R}}$ whose base-points are $p_{1}=[1:0:0]$, $p_{2}=[0:1:0]$, $p_{3}=[0:0:1]$ and which exchanges $p_{4}=[1:1:1]$ and $p_{5}=[a:b:c]$. Note that by blowing up $p_{1}, p_{2}, p_{3}, p_{4}, p_{5}$, the lift of $\lambda$ acts biregularly on $X$ and its induced action on the Picard group of $X$ with respect to the basis $\lbrace L,E_{1},E_{2},E_{3},E_{4},E_{5} \rbrace$ is given by the matrix \[ \small \begin{pmatrix} 2 & 1 & 1 & 1 & 0 & 0 \\ -1 & 0 & -1 & -1 & 0 & 0 \\ -1 & -1 & 0 & -1 & 0 & 0 \\ -1 & -1 & -1 & 0 & 0 & 0 \\ 0 & 0 & 0 & 0 & 0 & 1 \\ 0 & 0 & 0 & 0 & 1 & 0 \end{pmatrix}.\] So this automorphism of $X$ is in the kernel of $\rho$ and its action on the five pairs of conic bundle structures $\mathcal{R}_{i}$ is given by $\langle \gamma=(0,0,0,1,1) \rangle$. By changing the role of the $p_{i}$'s and proceeding in the same way, we get automorphisms of $X$ contained in $A_{0}$, whose induced actions on the $\mathcal{R}_{i}$'s are the ones of the elements $\gamma_{1}=(1,0,0,0,1)$, $\gamma_{2}=(0,1,0,0,1)$ and $\gamma_{3}=(0,0,1,0,1)$ in $(\Z/2\Z)^{5}$.

(\ref{it:item_(2)_Proposition_Aut(X)_Q22(4,0)}) We first show that there is no automorphism acting as a transposition on the pairs $\mathcal{R}_{i}$'s. Indeed, if there were an automorphism of $X$ of type $(12)$ for instance, we may assume that its action on the pairs $\mathcal{R}_{1}, \mathcal{R}_{2}$ is given by \[ (0) \, \myinline{
} \] by composing with the element of $A_{0}$ corresponding to $\gamma_{1}+\gamma_{2}=(1,1,0,0,0)$; and with respect to the action on the last three pairs $\mathcal{R}_{3}, \mathcal{R}_{4}, \mathcal{R}_{5}$ we may assume that its action is given either by \[ (1) \, \myinline{%
}, \] since we can compose with elements of $A_{0}$ corresponding to $\gamma_{3}=(0,0,1,0,1)$, $\gamma=(0,0,0,1,1)$ and $\gamma_{3}+\gamma=(0,0,1,1,0)$. So finally, we can focus on the following two cases: \begin{center} %
 \end{center}
In case (a), the induced action on the Picard group of $X$ with respect to the basis $\lbrace f_{1},f_{2},E_{p},E_{q},$ $E_{r},E_{s} \rbrace$ is such that $f_{1}, f_{2}, E_{p}, E_{s}$ are preserved, and $E_{q}$ and $E_{r}$ are exchanged; meaning that if such an automorphism exists, it comes from an automorphism of $\p^{1}_{\mathbb{R}} \times \p^{1}_{\mathbb{R}}$ that exchanges $q$ with $r$, and that fixes $p$ and $s$ as well as the two fibrations given by the two projections. Such an automorphism would be of the form $$(A,B) \in \PGL_{2}(\mathbb{R})^{2} \,\, \text{with} \,\, A = B = 
$$ and it would require $\mu_{1}=\mu_{2}=2$, which gives a contradiction on the parameters $\mu_{1}, \mu_{2}$ (see Lemma \ref{lem:Lemma_choice_of_points_Q22(4,0)}). Case (b) is not possible because the matrix of an action described in (b) on $\Pic(X)$ with respect to the basis $\lbrace f_{1},f_{2},E_{p},E_{q},E_{r},E_{s} \rbrace$ is \[ \small %
 \notin \GL_{6}(\mathbb{Z}),\] which shows that it is not even an automorphism of the Picard group. So, an element of type $(12)$ does not belong to the image of $\rho$.

We now study the existence of automorphisms acting as a double-transposition on the pairs of conic bundle structures. If there is an automorphism $\alpha$ of $X$ of type $(13)(45)$ for instance, we may assume that its action on the $\mathcal{R}_{i}$'s is one of the following \begin{center} %
 \end{center}
since we can compose with elements of $A_{0}$ that correspond in $(\Z/2\Z)^{5}$ to $\gamma_{1}+\gamma_{3}=(1,0,1,0,0)$, $\gamma=(0,0,0,1,1)$ and $\gamma_{1}+\gamma_{3}+\gamma=(1,0,1,1,1)$. In case (a), we see that $f_{1}$ and $f_{2}$ are exchanged by looking at pairs $\mathcal{R}_{4}$ and $\mathcal{R}_{5}$, and so $f_{1}+f_{2}$ is preserved. This implies that $E_{p}+E_{r}$ and $E_{q}+E_{s}$ are preserved by looking at the pair $\mathcal{R}_{2}$, and that $E_{p}+E_{q}$ with $E_{p}+E_{s}$ are interchanged and so are $E_{r}+E_{s}$ with $E_{r}+E_{q}$ by looking at the pairs $\mathcal{R}_{1}, \mathcal{R}_{3}$. It means that an automorphism of type (a), if it exists, comes from an automorphism $\alpha'$ of $\p^{1}_{\mathbb{R}} \times \p^{1}_{\mathbb{R}}$ exchanging the two fibrations, fixing the points $p$ and $r$, and exchanging $q$ with $s$. Such an automorphism is of the form $$\alpha' : (u,v) \mapsto (Bv,Au) \,\, \text{with} \,\, (A,B)=\left(
\right),$$ that is $\alpha' : ([u_{0}:u_{1}],[v_{0}:v_{1}]) \mapsto ([(\mu_{1}-1)v_{0}+v_{1}:\mu_{1}v_{1}),((\mu_{2}-1)u_{0}+u_{1}:\mu_{2}u_{1})$, and where $\mu_{1}, \mu_{2} \in \mathbb{R}\setminus\lbrace 0,1\rbrace$ satisfy the relation $\mu_{1}+\mu_{2}=1$. Case (b) is not possible because the action described in (b) does not even define an automorphism of the Picard group. Indeed, the matrix of its induced action on $\Pic(X)$ with respect to the basis $\lbrace f_{1},f_{2},E_{p},E_{q},E_{r},E_{s} \rbrace$ is \[ \small %
 \notin \GL_{6}(\mathbb{Z}).\] It shows that the element $(13)(45)$ does belong to the image of $\rho$ if and only if $\mu_{1}+\mu_{2}=1$.\label{par:paragraph_existence_autom_type_(13)(45)}

In addition, and for what remains, we also show that the element $(12)(45)$ belongs to the image of $\rho$ if and only if $\mu_{1}+\mu_{2}-\mu_{1}\mu_{2}=0$. Indeed, if there is an automorphism $\beta$ of $X$ of type $(12)(45)$, we may assume that its action on the pairs $\mathcal{R}_{i}$'s is one of the following two \begin{center} %
  \end{center} by composing with the elements of $A_{0}$ corresponding to $\gamma_{1}+\gamma_{2}=(1,1,0,0,0)$, $\gamma=(0,0,0,1,1)$ and $\gamma_{1}+\gamma_{2}+\gamma=(1,1,0,1,1)$. In both cases (a) and (b), $f_{1}$ and $f_{2}$ are exchanged, so $f_{1}+f_{2}$ is fixed (and so are $f_{1}+2f_{2}$ and $2f_{1}+f_{2}$). Looking at pairs $\mathcal{R}_{1}, \mathcal{R}_{2}$, we see that $E_{p}+E_{q}$ and $E_{p}+E_{r}$ are exchanged, and so are $E_{r}+E_{s}$ and $E_{q}+E_{s}$. In case (a), looking at the pair $\mathcal{R}_{3}$, we see that $E_{p}+E_{s}$ and $E_{q}+E_{r}$ are fixed. This implies that $E_{p}, E_{s}$ are fixed and that $E_{q}$ with $E_{r}$ are exchanged. So an automorphism of type $(12)(45)$ in case (a) comes from an automorphism $\beta'$ of $\p^{1}_{\mathbb{R}} \times \p^{1}_{\mathbb{R}}$ exchanging the two fibrations, fixing the points $p$ and $s$, and exchanging $q$ with $r$. The automorphism $\beta'$ is then given by $(x,y) \mapsto (By,Ax)$, where $(A,B) \in \PGL_{2}(\mathbb{R})^{2}$, and satisfies the conditions $\beta'(p)=p$, $\beta'(q)=r$, $\beta'(r)=q$ and $\beta'(s)=s$. This gives $$(A,B)=\left( 
 \right),$$ that is $\beta' : ([u_{0}:u_{1}],[v_{0}:v_{1}]) \mapsto ([v_{0}-v_{1}:-v_{1}],[u_{0}-u_{1}:-u_{1}])$, with the closed condition $\mu_{1}+\mu_{2}-\mu_{1}\mu_{2}=0$. Case (b) is not possible because the matrix of the action described in (b) on the Picard group of $X$, with respect to the basis $\lbrace f_{1},f_{2},E_{p},E_{q},E_{r},E_{s} \rbrace$, is \[ \small %
 \notin \GL_{6}(\mathbb{Z}),\] and this shows that it is not even an automorphism of $\Pic(X)$.

We continue explaining why there is no automorphism acting as a $3$-cycle on the pairs of conic bundle structures. If there were an automorphism $\alpha$ of $X$ of type $(123)$, for instance, its action on the pairs $\mathcal{R}_{4}, \mathcal{R}_{5}$ would be given either by \[ (1) \,\, \myinline{%
}, \] and we may assume that the action is as in (1) or as in (3) since we can compose an automorphism whose action is as in (2) or as in (4) with the element of $A_{0}$ corresponding to $\gamma=(0,0,0,1,1)$ to obtain an automorphism whose action is that in (1) and (3), respectively. With respect to the pairs $\mathcal{R}_{1}, \mathcal{R}_{2}, \mathcal{R}_{3}$, we may assume that its action is as in \[ (5) \,  \myinline{%
} \] by composing with elements of $A_{0}$ that correspond to $\gamma_{1}+\gamma_{2}=(1,1,0,0,0)$, $\gamma_{1}+\gamma_{3}=(1,0,1,0,0)$ and $\gamma_{2}+\gamma_{3}=(0,1,1,0,0)$. Summarising, we need to focus on the following cases: \begin{center}
%
 \end{center}
In both cases (a) and (c), we see that $f_{1}$ and $f_{2}$ are preserved by looking at pairs $\mathcal{R}_{4}, \mathcal{R}_{5}$, and hence $f_{1}+f_{2}$ is preserved. In case (a), looking at the action on pairs $\mathcal{R}_{1}, \mathcal{R}_{2}, \mathcal{R}_{3}$, we see that $E_{p}+E_{q} \mapsto E_{p}+E_{r} \mapsto E_{p}+E_{s} \mapsto E_{p}+E_{q}$ and $E_{r}+E_{s} \mapsto E_{q}+E_{s} \mapsto E_{q}+E_{r} \mapsto E_{r}+E_{s}$. 
This implies that $E_{p}$ is preserved and that $E_{q} \mapsto E_{r} \mapsto E_{s} \mapsto E_{q}$ are cyclically permuted. So an automorphism of type (a), if it exists, comes from an automorphism $\alpha'$ of $\p^{1}_{\mathbb{R}} \times \p^{1}_{\mathbb{R}}$ that preserves the two rulings, fixes $p=([1:0],[1:0])$, and permutes cyclically $q=([0:1],[0:1])$, $r=([1:1],[1:1])$, $s=([1:\mu_{1}],[1:\mu_{2}])$. The automorphism $\alpha'$ is then given by $(x,y) \mapsto (Ax,By)$, where $(A,B) \in \PGL_{2}(\mathbb{R})^{2}$ satisfies $$ A 
.$$ That would imply $$(A,B) = \left( %
 \right)$$ with $\mu_{1}, \mu_{2}$ satisfying $\mu_{1}^{2}-\mu_{1}+1=0$ and $\mu_{2}^{2}-\mu_{2}+1=0$, which has no real solutions, hence is not possible by Lemma \ref{lem:Lemma_choice_of_points_Q22(4,0)}. Case (c) is not even an automorphism of the Picard group since the matrix corresponding to an action described in (c) with respect to the basis $\lbrace f_{1},f_{2},E_{p},E_{q},E_{r},E_{s} \rbrace$ is \[ \small %
 \notin \GL_{6}(\mathbb{Z}).\] 
In both cases (b) and (d), we can note that $f_{2}$ is preserved and that $f_{1}$ together with $f_{1}+2f_{2}-E_{p}-E_{q}-E_{r}-E_{s}$ are exchanged by looking at pairs $\mathcal{R}_{4}$ and $\mathcal{R}_{5}$. In case (d), the induced action of such an automorphism on the Picard group of $X$ with respect to the basis $\lbrace f_{1},f_{2},E_{p},E_{q},E_{r},E_{s} \rbrace$ is given by the matrix \[ \small %
.\] So, if such an automorphism exists, it comes from a real birational map $\psi$ of $\p^{1}_{\mathbb{R}} \times \p^{1}_{\mathbb{R}}$ of order $6$ with $4$ base-points, $p$, $q$, $r$ and $s$, and which contracts $f_{2p}$ onto $p$, $f_{2s}$ onto $q$, $f_{2q}$ onto $r$ and $f_{2r}$ onto $s$. One can check by explicit computations that such a birational map exists only if $\mu_{1}^{2}-\mu_{1}+1=0$ and $\mu_{1}\mu_{2}-\mu_{1}+1=0$, which leads to a contradiction on the real parameters $\mu_{1}, \mu_{2}$ (see Lemma \ref{lem:Lemma_choice_of_points_Q22(4,0)}). Case (b) is not possible because the matrix of the action described in (b) on the Picard group of $X$ with respect to the basis $\lbrace f_{1},f_{2},E_{p},E_{q},E_{r},E_{s} \rbrace$ is \[ \small 
 \notin \GL_{6}(\mathbb{Z}),\] which shows that it is not even an automorphism of $\Pic(X)$.

Now we show that there is no automorphism acting as a $4$-cycle on the five pairs of conic bundle structures. Indeed, if there were an automorphism of $X$ of type $(1435)$ for instance, we may assume that its action on the $\mathcal{R}_{i}$'s is one of the following four \begin{center} %
 \end{center} 
since we can compose with the elements of $A_{0}$ corresponding to $\gamma_{1}+\gamma_{3}+\gamma=(1,0,1,1,1)$, $\gamma_{3}+\gamma=(0,0,1,1,0)$, $\gamma_{3}=(0,0,1,0,1)$, $\gamma=(0,0,0,1,1)$, $\gamma_{1}+\gamma_{3}=(1,0,1,0,0)$, $\gamma_{1}+\gamma=(1,0,0,1,0)$ and $\gamma_{1}=(1,0,0,0,1)$ to bring us back to automorphisms of $X$ whose action on the pairs is that in (a), (b), (c) and (d). In case (a), the induced action of such an automorphism on the Picard group of $X$ with respect to the basis $\lbrace f_{1},f_{2},E_{p},E_{q},E_{r},E_{s} \rbrace$ is given by the matrix \small\[ \begin{pmatrix} 1 & 1 & 1 & 0 & 0 & 1 \\ 1 & 1 & 1 & 1 & 0 & 0 \\ -1 & -1 & -1 & -1 & 0 & -1 \\ 0 & -1 & -1 & 0 & 0 & 0 \\ 0 & 0 & 0 & 0 & 1 & 0 \\ -1 & 0 & -1 & 0 & 0 & 0 \end{pmatrix}.\] Then if such an automorphism exists, it comes from a real birational map $\psi$ of $\p^{1}_{\mathbb{R}} \times \p^{1}_{\mathbb{R}}$ of bidegree $(1,1)$ of order $4$, with $3$ base-points $p, q, s$, and which fixes the point $r$. It is such that $\psi^{2}$ induces an automorphism of $X$ of type $(13)(45)$, which exists if and only if $\mu_{1}+\mu_{2}=1$ as we saw previously. Since the rational inverse of $\psi$ contracts $f_{2p}$ onto $q$, $f_{1p}$ onto $s$ and $f_{pqs}$ onto $p$, one deduces that $\psi$ contracts the curves $f_{1p}$ onto $q$, $f_{2p}$ onto $s$ and $f_{pqs}$ onto $p$. One can check by explicit computations that such a birational map exists exactly if $\mu_{1}+\mu_{2}=1$ and $2\mu_{1}^{2}-2\mu_{1}+1=0$, which gives an impossible condition on the real parameters $\mu_{1}, \mu_{2}$ (see Lemma \ref{lem:Lemma_choice_of_points_Q22(4,0)}). Cases (b) and (c) are not possible either because the matrices corresponding to the actions described in (b) and (c) on the Picard group of $X$, with respect to the basis $\lbrace f_{1},f_{2},E_{p},E_{q},E_{r},E_{s} \rbrace$, are \[ \small \begin{pmatrix} 1 & 1 & 1 & 0 & 0 & 1 \\ 1 & 1 & 1 & 1 & 0 & 0 \\ -1 & -1 & -\frac{3}{2} & -\frac{1}{2} & -\frac{1}{2} & -\frac{1}{2} \\ 0 & -1 & -\frac{1}{2} & -\frac{1}{2} & \frac{1}{2} & -\frac{1}{2} \\ 0 & 0 & -\frac{1}{2} & \frac{1}{2} & \frac{1}{2} & \frac{1}{2} \\ -1 & 0 & -\frac{1}{2} & -\frac{1}{2} & \frac{1}{2} & -\frac{1}{2} \end{pmatrix} \, , \, \small \begin{pmatrix} 1 & 1 & 1 & 0 & 0 & 1 \\ 1 & 1 & 0 & 0 & 1 & 1 \\ -1 & -1 & -\frac{1}{2} & -\frac{1}{2} & -\frac{1}{2} & -\frac{3}{2} \\ 0 & -1 & -\frac{1}{2} & \frac{1}{2} & -\frac{1}{2} & -\frac{1}{2} \\ 0 & 0 & \frac{1}{2} & \frac{1}{2} & \frac{1}{2} & -\frac{1}{2} \\ -1 & 0 & -\frac{1}{2} & \frac{1}{2} & -\frac{1}{2} & -\frac{1}{2} \end{pmatrix} \notin \GL_{6}(\mathbb{Z})\]
respectively, which shows that these are not even automorphisms of the Picard group. In case (d), the induced action on the Picard group of $X$ with respect to the basis $\lbrace f_{1},f_{2},E_{p},E_{q},E_{r},E_{s} \rbrace$ is given by the matrix \small\[ \begin{pmatrix} 1 & 1 & 1 & 0 & 0 & 1 \\ 1 & 1 & 0 & 0 & 1 & 1 \\ -1 & -1 & -1 & 0 & -1 & -1 \\ 0 & -1 & 0 & 0 & 0 & -1 \\ 0 & 0 & 0 & 1 & 0 & 0 \\ -1 & 0 & 0 & 0 & 0 & -1 \end{pmatrix}.\] So, if such an automorphism exists, it comes from a real birational map $\varphi$ of $\p^{1}_{\mathbb{R}} \times \p^{1}_{\mathbb{R}}$ of order 8 with $3$ base-points $p, r, s$, which sends the point $q$ to $r$, and which contracts the curves $f_{prs}$ onto $p$, $f_{1s}$ onto $q$ and $f_{2s}$ onto $s$. One can check by explicit computations that a birational transformation with such characteristics exists only if $\mu_{1}+\mu_{2}=1$ and $2\mu_{2}^{2}-2\mu_{2}+1=0$, which is not possible due to the choice of the parameters $\mu_{1}, \mu_{2}$.

Finally, we study the existence of automorphisms acting as a $5$-cycle on the pairs of conic bundle structures. For instance, the action of an automorphism $\alpha$ of $X$ of type $(13245)$ on the $\mathcal{R}_{i}$'s may be assumed to be of the form \begin{center} 
 \end{center}
since we can compose with elements in the kernel of $\rho$ (with three $0$ and two $1$, or with one $0$ and four $1$) to obtain an automorphism of $X$ whose action on the pairs is that in (a) or (b). In case (a), the matrix of the induced action of such an automorphism $\Pic(X)$ with respect to the basis $\lbrace f_{1},f_{2},E_{p},E_{q},E_{r},E_{s} \rbrace$ is \[\small %
.\] So, if such an automorphism exists, it comes from a real birational map $\psi$ of $\p^{1}_{\mathbb{R}} \times \p^{1}_{\mathbb{R}}$ of order $5$ with $2$ base-points, $p$ and $r$, which sends the points $q$ to $s$ and $s$ to $r$, and which contracts the curves $f_{1r}$ onto $p$ and $f_{1p}$ onto $q$. We check by explicit computations that such a birational transformation $\psi$ exists exactly if $\left\{ \mu_{1}=\frac{1-\sqrt{5}}{2} , \mu_{2}=\frac{3-\sqrt{5}}{2} \right\}$ or if $\left\{ \mu_{1}=\frac{1+\sqrt{5}}{2} , \mu_{2}=\frac{3+\sqrt{5}}{2} \right\}$, and it is given by the lift of the birational maps \begin{align*} \psi_{1} : ([u_{0}:u_{1}],[v_{0},v_{1}]) &\dashmapsto ([2(v_{1}-v_{0})u_{1}:(\sqrt{5}-1)v_{1}(u_{0}-u_{1})],[2u_{1}:(\sqrt{5}-3)(u_{0}-u_{1})]); \\ \psi_{2} : ([u_{0}:u_{1}],[v_{0},v_{1}]) &\dashmapsto ([-2(v_{1}-v_{0})u_{1}:(\sqrt{5}+1)v_{1}(u_{0}-u_{1})],[2u_{1}:-(\sqrt{5}+3)(u_{0}-u_{1})]),
\end{align*} respectively. Case (b) is not possible because the matrix of the action described in (b) on the Picard group of $X$, with respect to the basis $\lbrace f_{1},f_{2},E_{p},E_{q},E_{r},E_{s} \rbrace$, is \[ \small \begin{pmatrix} 2 & 1 & 1 & 1 & 1 & 1 \\ 1 & 1 & 0 & 1 & 0 & 1 \\ -1 & 0 & -\frac{1}{2} & -\frac{1}{2} & \frac{1}{2} & -\frac{1}{2} \\ -1 & 0 & \frac{1}{2} & -\frac{1}{2} & -\frac{1}{2} & -\frac{1}{2} \\ -1 & -1 & -\frac{1}{2} & -\frac{1}{2} & -\frac{1}{2} & -\frac{3}{2} \\ -1 & -1 & -\frac{1}{2} & -\frac{3}{2} & -\frac{1}{2} & -\frac{1}{2} \end{pmatrix} \notin \GL_{6}(\Z),\] which shows that it is not even an automorphism of $\Pic(X)$.          

\end{proof}

\begin{proof}[Proof of Theorem \ref{Theorem:Theorem_main_result}]
The result is a consequence of Lemmas \ref{lem:Lemma_choice_of_points_Q31(0,2)_Rob}, \ref{lem:Lemma_choice_of_points_Q31(2,1)}, \ref{lem:Lemma_choice_of_points_Q31(4,0)}, \ref{lem:Lemma_choice_of_points_Q22(0,2)}, \ref{lem:Lemma_choice_of_points_Q22(4,0)}, and Propositions \ref{Prop:Proposition_Aut(X)_Q31(0,2)}, \ref{Prop:Proposition_Aut(X)_Q31(2,1)}, \ref{Prop:Proposition_Aut(X)_Q31(4,0)}, \ref{Prop:Proposition_Aut(X)_Q22(0,2)} and \ref{Prop:Proposition_Aut(X)_Q22(4,0)}.
\end{proof}


\Addresses

\end{document}